\documentclass[3p,times]{elsarticle}

\usepackage{ecrc}
\usepackage{color}

\volume{00}

\firstpage{1}

\journalname{}

\runauth{M. Cai, S.Gan and Y. Hu}

\jid{anms}

\jnltitlelogo{}

\CopyrightLine{2011}{Published by Elsevier Ltd.}

\usepackage{amssymb}
\usepackage{amsthm,amsmath}

\usepackage[figuresright]{rotating}

\newcommand{\R}{\mathbb{R}}
\newcommand{\N}{\mathbb{N}}
\newcommand{\E}{\mathbb{E}}

\newcommand{\dd}{\mathrm{d}}

\newtheorem{assumption}{Assumption}[section]
\newtheorem{lemma}{Lemma}[section]
\newtheorem{theorem}{Theorem}[section]
\newtheorem{corollary}{Corollary}[section]
\newtheorem{proposition}{Proposition}[section]
\newtheorem{remark}{Remark}[section]

\begin{document}

\begin{frontmatter}



\dochead{}

\title{Weak convergence of the backward Euler method for stochastic {C}ahn--{H}illiard equation with additive noise}


 \author[label1]{Meng Cai}\ead{mcai1993@126.com}
 \author[label1]{Siqing Gan \corref{cor1}}\ead{sqgan@csu.edu.cn}
 \author[label2]{Yaozhong Hu}\ead{yaozhong@ualberta.ca}
 \address[label1]{School of Mathematics and Statistics, HNP-LAMA,
Central South University, 410083, Hunan, China}
 \address[label2]{Department of Mathematical and Statistical Sciences,
University of Alberta, T6G 2G1, Edmonton, Canada}
\cortext[cor1]{Corresponding author}

\begin{abstract}
We prove a weak rate of convergence of a fully discrete scheme for
  stochastic Cahn--Hilliard equation with  additive noise,
  where the spectral Galerkin method is used in space and the backward Euler method is used in time.
  Compared with the Allen--Cahn type stochastic partial differential equation, the error analysis here is much more sophisticated due to the presence of the unbounded operator in front of the nonlinear term.
  To address such issues, a novel and direct approach has been exploited which does not rely on a Kolmogorov equation but on the integration by parts formula from Malliavin calculus.
   To the best of our knowledge, the rates of weak convergence
    are revealed in the stochastic Cahn--Hilliard equation setting for the first time.
\end{abstract}


\end{frontmatter}


\section{Introduction}	

During the last decades, there have been  overwhelming activities on the analysis of numerical stochastic partial differential equation (SPDE) under globally Lipschitz condition and a fast growing number of studies on Allen--Cahn type SPDE with non-globally Lipschitz coefficients.
However, numerical analysis of stochastic Cahn--Hilliard equation, which is another prominent SPDE model with non-globally Lipschitz coefficients, is at its beginning and is far from being well understood.
The Cahn--Hilliard equation is of fundamental importance in various applications to, such as, the complicated phase separation and coarsening phenomena in a melted alloy \cite{cahn1961on,cahn1971spinodal},
spinodal decomposition for binary mixture \cite{cahn1958free},
the diffusive process of populations and oil film spreading over a solid surface \cite{cho1997asympotic}.
Our motivating example arises from a simplified  mesoscopic physical model for phase separation.
The aim of this article is to investigate the weak convergence rate of a full discretization for stochastic Cahn--Hilliard equation driven by additive noise,
\begin{equation}\label{eq:CHC-abstract}
\left\{
    \begin{array}{lll}
    \dd X(t) + A( A X(t) + F ( X(t) ))
    \, \dd t
    =
    \dd W(t), \quad  t \in (0, T],
    \\
     X(0) = X_0.
    \end{array}\right.
\end{equation}

Let  $\mathbf{D}$ be a bounded connected open domain of $\R^d, d=1,2,3$ with smooth boundary and let  $H:= L^2 (  \mathbf{D}, \R )$ be the Hilbert space
with the  usual scalar product
$\langle \cdot ,  \cdot \rangle$ and norm $\|\cdot\|$.
The space $\dot{H} := \{ v \in H: \int_{\mathbf{D}} v \dd x =0\}$ is a subspace of $H$.
We make the following  assumptions.
\begin{assumption}\label{assum:linear-operator-A}
  $-A:\mathrm{dom}(A) \subseteq \dot{H} \to \dot{H}$  is the Neumann Laplacian defined by $-Au = \Delta u, u \in \mathrm{dom}(A)
  = \{ u \in H^2( \mathbf{D}) \cap \dot{H}: \frac {\partial u} {\partial n}=0 \,\, \mathrm{on} \,\, \partial \mathbf{D}\}$.
\end{assumption}

\begin{assumption}\label{assum:nonlinearity}
$F : L^6 (  \mathbf{D} , \mathbb{R} ) \rightarrow H$
is the Nemytskii operator given by
\begin{equation}
F (v)(x)
=f ( v(x))
=v^3(x)-v(x),
\quad
x \in  \mathbf{D}, v \in L^6(  \mathbf{D} , \mathbb{R}).
\end{equation}
\end{assumption}

\begin{assumption}\label{assum:noise-term}
The noise process $\{ W(t)\}_{t \in [0,T]}$ is an $\dot{H}$-valued $Q$-Wiener process with the covariance operator $Q$ satisfying
\begin{equation}\label{eq:ass-AQ-condition}
\big\| A^{\frac12}  Q^{\frac12}  \big\|_{\mathcal{L}_2} < \infty.
\end{equation}
\end{assumption}

\begin{assumption}\label{assum:intial-value-data}
The initial value $X_0$  is deterministic
and satisfies
\begin{equation}\label{eq:X0}
|X_0|_4 < \infty,
\end{equation}
where the norm $|\cdot|_4$ is defined in \eqref{eq:norm} below.
\end{assumption}

We point out that Assumption \ref{assum:noise-term} is the same as that in \cite{furihata2018strong,kovacs2011finite,qi2020error}.
The assumption on the initial datum can be relaxed,
but at the expense of having the constant $C$  depending on
$t^{-1}$, by exploiting the smoothing effect of the semigroup
$E(t), t \in [0, T ]$ and standard non--smooth data error estimates.

Based on the above assumptions and following the semigroup framework in \cite{prato2014stochastic},
we see that the model \eqref{eq:CHC-abstract} admits a unique mild solution
\begin{align*}\label{eq:mild-solution-introduction}
 X(t) =  E(t) X_0  -  \int_0^t E(t-s)
     A P F(X(s)) \,\mathrm{d} s
         +  \int_0^t E(t-s) \, \mathrm{d} W(s),
         \quad  t \in [0, T],
\end{align*}
where $E(t)$ denotes the analytic semigroup generated by $-A^2$.
We refer the readers  to
\cite{antonopoulou2016existence,cardon2001cahn,cui2019wellposedness,cui2020absolute,prato1996stochastic,elezovic1991on,qi2021existence}
 for  the existence and uniqueness of the mild solution for such equation.
Since the exact solutions are rarely known explicitly,
numerical simulations are often used to investigate the behavior of the solutions.
We choose the spatial semi-discretization by the spectral Galerkin method, i.e., projecting the equation to vector space $H_N$, spanned by the first $N$ eigenvectors of $A$.
The approximated equation of \eqref{eq:CHC-abstract} is in the form
\begin{equation*}
\dd X^N(t) + A (A X^N(t)  + P_N F( X^N(t) )) \dd t
=
P_N \dd W(t),\,
t \in (0,T];\, \,  X^N(0)=P_N X_0\,,
\end{equation*}
where $P_N$ is the spectral Galerkin projection operator onto the space $H_N$.
In the temporal direction, we apply the backward Euler method to the above equation.
The fully discrete scheme is then given by
\begin{equation*}
X_{t_m}^{M,N} - X_{t_{m-1}}^{M,N} + \tau A^2 X_{t_m}^{M,N}
   + \tau P_N A F(X_{t_m}^{M,N})  = P_N \Delta W_m, \quad
   m \in \{ 1,2,\cdots,M \}.
\end{equation*}
Here $\Delta W_m := W(t_m)-W(t_{m-1})$, $ \tau = \tfrac TM$ is the time stepsize and $t_m = m \tau$.
The main result, concerning the weak convergence rates of the full discretization, reads
\begin{equation}\label{intro:weak}
\big| \E [\Phi(X(T))] - \E [\Phi(X_T^{M,N})] \big|
 \leq C   \big(  \lambda_N^{-2} +
 \tau \big), \,
 \forall \, \Phi \in C_b^2(\dot{H},\R).
\end{equation}
Here and throughout this article, $C$ denotes a generic positive constant that is independent of the discretization parameters $M,N$ and may change from line to line and $C_b^2(\dot{H},\R)$ (or $C_b^2$) represents the space of not necessarily bounded mappings from $\dot{H}$ to $\R$ that have continuous and bounded Fr\'{e}chet derivatives up to order 2.
We split the weak error into two terms, both the spatial error and the temporal error, which are analyzed in Section 3 and Section 4, respectively.
The result given by the above inequality \eqref{intro:weak} is on the weak rate of convergence.
It is strictly greater than the strong ones (see Corollary \ref{coro:strong}) as expected.
It is seen that the weak rate (which is 1.0 in time) is not twice as the strong one, contrary to the common belief.
Indeed, the order is limited to $1$ since an implicit Euler scheme is used.

The idea for error analysis to obtain \eqref{intro:weak} goes as follows.
At first, the weak error is separated into two parts,
 the spatial error and the temporal error,
\begin{equation}
 \E\big[ \Phi(X(T))\big]
-\E \big[\Phi(X_T^{M,N})\big]
\\
=
\big(
\E\big[\Phi(X(T))\big]
-\E \big[\Phi(X^N(T))\big]
\big)
+
\big(
\E\big[\Phi(X^N(T))\big]
-\E \big[\Phi(X_T^{M,N})\big]
\big).
\end{equation}
To simplify the notation, we often write $\mathcal O_t$ for  $\int_0^t E(t-r)\dd W(r)$
and $\mathcal O^N_t = P_N \mathcal O_t$.
By introducing two processes
$\bar{X}(t):=X(t)-\mathcal{O}_t$
and
$\bar{X}^N(t):=X^N(t)-\mathcal{O}_t^N$,
we can further split  the spatial error as
\begin{equation} \label{eq:intro-spatial-error-decomposition}
\begin{split}
\E \big[  \Phi(X(T)) \big] - \E \big[ \Phi(X^N(T)) \big]
  & =  \big(\E\big[\Phi(\bar{X}(T)+\mathcal{O}_T)\big]
      -   \E \big[\Phi(\bar{X}^N(T)+\mathcal{O}_T)\big]\big) \\
  & \quad +   \big( \E \big[ \Phi(\bar{X}^N(T)+\mathcal{O}_T) \big]
      -\E \big[\Phi(\bar{X}^N(T)+\mathcal{O}_T^N)\big]\big).
\end{split}
\end{equation}
To proceed,
one relies on the Taylor expansion of the test function $\Phi$.
The key argument to estimate the first term on the right hand of
\eqref{eq:intro-spatial-error-decomposition} is to bound the  error between $\bar{X}^N(T)$ and $\bar{X}(T)$ by that in a strong sense,
\begin{align}
\begin{split}
&
\big|
\E\big[\Phi(\bar{X}(T)+\mathcal{O}_T)\big]
-\E \big[\Phi(\bar{X}^N(T)+\mathcal{O}_T)\big]
\big|
\\
& \quad \leq C  \Big|
\E \int_0^1 \Phi'  \big(  X(T)  +  \lambda
( \bar{X}^N(T) - \bar{X}(T) ) \big)
\big( \bar{X}^N(T) - \bar{X}(T) \big) \dd \lambda \Big|
\\&
\quad \leq
C\,
\| \bar{X}(T)-P_N \bar{X}(T)\|_{L^2(\Omega,\dot{H})}
+
C\,
\| P_N \bar{X}(T)- \bar{X}^N(T) \|_{L^2(\Omega,\dot{H})}
.
\end{split}
\end{align}
The error term $\| \bar{X}(T)-P_N \bar{X}(T)\|_{L^2(\Omega,\dot{H})}$ can be easily controlled owing to the higher spatial regularity of the stochastic process $\bar{X}(T)$,
in the absence  of the stochastic convolution.
The remaining term  $ e^N(t) := P_N \bar{X}(t) - \bar{X}^N(t)$,
satisfying the following random PDE,
\begin{equation}\label{eq:intro-randomPDE}
\tfrac{\dd}{\dd t}  e^N(t)
  + A^2  e^N(t) + P_N A \big[ F(X(t)) - F(X^N(t)) \big] =0,
  \quad  e^N(0)=0,
\end{equation}
must be carefully treated due to the presence of the unbounded operator $A$ before the nonlinear term $F$.
We use  the monotonicity of the nonlinearity of $F$ and the regularities of $X(T)$, $X^N(T)$ and $\mathcal{O}_t$ to  derive
$\Big\| \int_0^T |  e^N(t) |_1^2 \dd t \Big\|_{L^p(\Omega,\R)}
     \leq  C \, \lambda_N^{-4} $.
Then, combining it with the mild  solution of \eqref{eq:intro-randomPDE} leads to the desired weak orders (c.f. \eqref{eq:e(T)}-\eqref{eq:e(T)-order} below).
Subsequently, we turn our attention to the second term in \eqref{eq:intro-spatial-error-decomposition}.
Applying the Taylor expansion gives
\begin{align}
\begin{split}
& \big|
\E\big[\Phi(\bar{X}^N(T)+\mathcal{O}_T)\big]
-\E \big[\Phi(\bar{X}^N(T)+\mathcal{O}_T^N)\big]
\big|
\leq
\Big| \E \big[ \Phi^{'} (X^N(T))(\mathcal{O}_T-\mathcal{O}_T^N) \big] \Big|
\\ &  \quad +
\Big |  \E \Big [ \int_0^1 \Phi^{''}(X^N(T)+\lambda(\mathcal{O}_T-\mathcal{O}_T^N))
(\mathcal{O}_T-\mathcal{O}_T^N,\mathcal{O}_T-\mathcal{O}_T^N)
(1-\lambda)\dd \lambda \Big] \Big|.
\end{split}
\end{align}
The Malliavin integration by parts formula is the key ingredient to deal with the first term (c.f. \eqref{eq:space-rate-Malliavin}) and the second term can be easily estimated due to the boundedness of $\Phi''$.
 It is now easy to explain why the weak rate of convergence is expected to be higher than strong convergence rate.
As a byproduct of the weak error analysis, one can easily obtain the rate of the strong error,
\begin{equation}
\|X(t)-X^N(t)\|_{L^2(\Omega,\dot{H})}
\leq
\|\bar{X}(t)-\bar{X}^N(t)\|_{L^2(\Omega,\dot{H})}
+
\|\mathcal{O}_t-\mathcal{O}^N_t\|_{L^2(\Omega,\dot{H})}
\leq C \lambda_N^{-\frac 32},
\end{equation}
which is consistent with the results in \cite{cui2021strongCHC,qi2021existence} and is  lower than the weak convergence rate in \eqref{intro:weak}, due to the presence of the second error.
The basic idea to estimate temporal error  is the same as that of the spatial error by essentially exploiting the discrete analogue of  the arguments.
The main point is that error must be uniform on  the spatially discrete parameter $N$.

Having sketched the central ideas of the weak error analysis,
we  review some   relevant results in the literature.
For the linearized stochastic Cahn--Hilliard equations,  we refer to
\cite{chai2018conforming,kossioris2013finite,larsson2011finite}
for  some strong convergence results of the finite element method.
The authors in \cite{furihata2018strong,kovacs2011finite} studied the strong convergence of the fully discrete finite element approximation for Cahn--Hilliard--Cook equation under spatial regular noise, but with no rates obtained.
Later, the authors in \cite{qi2020error} derives  strong convergence rates of the mixed finite element method  by using a priori strong moment bounds of the numerical approximations.
For unbounded noise diffusion,
the existence and regularity of solution have been investigated in
\cite{antonopoulou2016existence,cui2019wellposedness} and the absolute continuity has been studied in
\cite{antonopoulou2018malliavin,cui2020absolute}.
Recently, the strong convergence rates  of the  spatial spectral Galerkin method and  the  temporal accelerated implicit Euler method for the stochastic Cahn--Hilliard equation were obtained in \cite{cui2021strongCHC}.
For weak convergence analysis in the non--globally Lipschitz setting,
we are only aware of the papers
\cite{brehier2019weak,cai2021weak,cui2019strong,cui2021weak} concerning the stochastic Allen--Cahn equation.
To the best of our knowledge, the weak convergence rates of a fully discrete method for the stochastic Cahn--Hilliard equation are absent in the literature.
It is worthwhile to  point out that issues from the presence of the unbounded operator  in front of the nonlinear term make the weak error  analysis  much more challenging.
To be more specific,  in addition to the aforementioned difficulty in the weak analysis,  the estimate of the Malliavin derivative for the spatial approximation  process is also completely different,
much  more efforts are needed  (c.f. Proposition
\ref{Estimate of Malliavin derivative of the solution}).
More recently, while this work was under review, we were aware of the preprint \cite{brehier2022weak} posted in arXiv, concerning with numerical approximations of similar SPDEs,
where Br\'{e}hier, Cui and Wang provide weak error estimates for another class of numerical schemes, whose  weak order is twice as the strong
order, for less regular problems.
It is worth mentioning that the approach in the two works are substantially different.
Different methods and different regularity regimes are dealt with.

The outline of the article is  as follows.
In the next section, we present some preliminaries, including  the  well-posedness and regularity of the mild solution and give a brief introduction to   Malliavin calculus.
Section \ref{sec:spatial-weak-analysis} is devoted to the weak analysis of the spectral Galerkin method in space and Section \ref{sec:full-weak-analysis} is concerned with the weak convergence rates of the backward Euler method in time.

\section{Preliminaries}\label{sec:preliminaries}

In this section, the mathematical setting, well-posedness and regularity of the model and a brief introduction to Malliavin calculus are given.

\subsection{Mathematical setting}

Given two  real separable Hilbert spaces
$(H, \langle \cdot, \cdot \rangle, \|\cdot\| )$
and
$(U, \langle \cdot, \cdot \rangle_U, \|\cdot\|_U )$,
$\mathcal{L}(U,H)$
stands for the space of all bounded linear operators from
$U$ to $H$  with the operator norm
$\| \cdot \|_{\mathcal{L}(U,H)}$
and
$\mathcal{L}_2(U,H) (\subset \mathcal{L}(U,H)$)
denotes  the space   of all Hilbert-Schmidt operators from $U$ to $H$.
For simplicity, we  write $\mathcal{L}(H) $ and $\mathcal{L}_2(H)$ (or $\mathcal{L}_2$ for short) instead of $\mathcal{L}(H,H)$ and $\mathcal{L}_2(H,H)$, respectively.
 It is known, see e.g., \cite{prato2014stochastic}, that $\mathcal{L}_2 (U,H)$ is a Hilbert space equipped with the inner product and norm,
\begin{align}
	\left< T_1 , T_2 \right>_{\mathcal{L}_2(U,H)}
	=
	\sum_{i\in\N^{+}} \left< T_1 \phi_i , T_2 \phi_i \right>,
	\;
	\| T \|_{\mathcal{L}_2(U,H)}
	=\Big(
	\sum_{i \in \N^{+}} \| T \phi_i \|^2
	\Big)^{\frac12},  \label{e.2.1}
\end{align}
where
$\{\phi_i\}$ is   an arbitrary  orthonormal basis of $U$.
Let   $H=L^2(\mathbf{D},\R)$ and
$\dot{H}=\{v \in H: \langle v,1 \rangle=0\}$.
$V:=C( \mathbf{D},\R)$ denotes
the Banach space of all continuous functions with
 supremum norm $\|\cdot\|_V$ and
$ L^r( \mathbf{D} , \mathbb{R}):=\{f:  \mathbf{D} \to \R,   \int_\mathbf{D} |f(x)|^r dx<\infty\}$.
We define
$P:H\rightarrow \dot{H}$
the generalized orthogonal projection by
$P v= v - |  \mathbf{D} |^{-1}
\int_{\mathbf{D}} v \dd x$, then
$(I-P)v= |  \mathbf{D} |^{-1}
\int_{\mathbf{D}} v \dd x$
is the average of $v$ over $\mathbf{D}$.


It is easy to check that $A$ is a positive definite, self-adjoint and unbounded linear operator on $\dot{H}$ with compact inverse.
For any $v \in H$, we define $A v=A P v$, then
there exists  a family of eigenpairs
$\{e_j , \lambda_j \}_{j \in \N }$
such that
\begin{align}\label{eq:eigenvalue-A-H}
	A e_j=\lambda_j e_j
	\quad \text{and} \quad
	0=
	\lambda_0
	<
	\lambda_1
	\leq
	\lambda_2
	\leq
	\cdots
	\leq
	\lambda_j
	\leq\cdots
	\quad \text{with} \quad
	\lambda_j \rightarrow \infty,
\end{align}
where
$e_0= |\mathbf{D} |^{-\frac 12}$
and
$\{ e_j, j=  1, \cdots \} $
forms an orthonormal basis of $\dot{H}$.
Straightforward applications of the spectral theory yield the fractional powers of $A$  on $\dot{H}$,
e.g.,
$A^\alpha v
=
\sum_{j=1}^\infty
\lambda_j^\alpha
\langle v,e_j \rangle e_j
$,
$\alpha \in \mathbb{R}$,
$v \in \dot{H}$.
The space $\dot{H}^\alpha = \mathrm{dom} (A^{\frac\alpha2}),
	\alpha \in \R$
is a Hilbert space with the inner product
$\langle \cdot, \cdot \rangle_{\alpha}$
and the associated norm
$|\cdot|_\alpha $
given by
\begin{equation}\label{eq:norm}
	\langle v , w \rangle_{\alpha}
	=
	\sum_{j=1}^{\infty} \lambda_j^{\alpha}
	\langle v , e_j \rangle
	\langle w , e_j \rangle,
	\quad
	|v|_{\alpha}= \| A^{\frac\alpha 2} v \|=
	\Big(
	\sum_{j=1}^{\infty} \lambda_j^{\alpha}
	|\langle v,e_j\rangle |^2
	\Big)^{\frac12}.
\end{equation}
We also define
$
	\|u\|_{\alpha}
	=
	\big(
	|u|_{\alpha}^2+|\langle u,e_0 \rangle|^2
	\big)^{\frac12}
$ for $u \in H$
and the corresponding space is
$H^{\alpha}:=
	\{u \in H:\|u\|_{\alpha} < \infty \}.$
A basic fact shows that for  $\alpha =1,2$,
the norm $|\cdot|_\alpha$ on $\dot{H}^\alpha$
is  equivalent to the standard Sobolev norm
$\| \cdot \|_{H^\alpha( \mathbf{D})}$
(see \cite[Theorems 2.9, 2.12]{kim2020fractional} and
\cite[Theorem 16.9]{yagi2010abstract}).
Since $H^2(\mathbf{D})$ is an algebra, there is a constant
$C >0$ such that, for any $f,g \in \dot{H}^2$,
\begin{equation}\label{eq:norm-algebra}
	\|fg\|_{H^2( \mathbf{D})}
            \leq C \|f\|_{H^2( \mathbf{D})}
            \|g\|_{H^2( \mathbf{D})}
	        \leq C |f|_{2} |g|_{2}.
\end{equation}
We recall that the operator $-A^2$   generates an analytic semigroup $E(t)=e^{-tA^2}$
on $H$ due to \eqref{eq:eigenvalue-A-H} and we have
\begin{align}\label{eq:definition-semigroup-E(t)}
	\begin{split}
		E(t)v=e^{-tA^2}v
		=Pe^{-t A^2}v
		+
		(I-P)v,
		\quad
		v \in H.
	\end{split}
\end{align}
With the aid of the eigenbasis of $A$ and Parseval's identity, we have
\begin{align}
	\| A^\mu E(t) \|_{\mathcal{L}(\dot{H})}
	&
	\leq
	C t^{-\frac\mu2} , \; t>0 , \; \mu \geq 0,
	\label{I-spatial-temporal-S(t)}
	\\
	\| A^{-\nu}(I-E(t)) \|_{\mathcal{L}(\dot{H})}
	&
	\leq
	C t^{\frac\nu2} , \quad t \geq 0, \; \nu \in[0,2],
	\label{II-spatial-temporal-S(t)}
	\\
	\int_{t_1}^{t_2} \| A^\varrho E(s) v \|^2 \, \dd s
	&
	\leq
	C |t_2-t_1|^{1-\varrho} \| v \|^2, \; \forall v \in \dot{H},
	\varrho \in [0,1],
	\label{III-spatial-temporal-S(t)}
	\\
	\Big\| A^{2\rho} \int_{t_1}^{t_2}
	E(t_2-\sigma)v \, \dd \sigma
	\Big\|
	&
	\leq
	C |t_2-t_1|^{1-\rho} \| v \|, \; \forall v \in \dot{H},
	\rho \in [0,1].
	\label{IV-spatial-temporal-S(t)}
\end{align}

By  Assumption \ref{assum:nonlinearity},
there exists a constant $C>0$ such that
\begin{align}
	- \langle F(u)-F(v) , u-v \rangle
	& \leq  \| u-v \|^2,
	\quad u  , v \in L^6(  \mathbf{D},\mathbb{R}),
	\label{eq:one-side-condition}
	\\  \| F (u) - F (v) \|  &\leq
	C ( 1 + \| u \|_V^2 +  \| v \|_V^2 ) \| u - v \|,
	\quad u, v \in V.
	\label{eq:local-condition}
\end{align}

\subsection{Well-posedness and regularity results of the model}

First at all, similar to \cite[(2.5) $\& \,\,(2.7)$]{cui2021strongCHC}, we give the following lemma concerning the spatio-temporal regularity result of stochastic convolution $\mathcal{O}_t := \int_0^t  E(t-s)  \dd W(s)$.

\begin{lemma}\label{lem:stochastic-convolution}
Suppose Assumptions \ref{assum:linear-operator-A} and \ref{assum:noise-term} hold.
Then for all $ p \ge 1$,
the stochastic convolution $\mathcal{O}_t$ satisfies
\begin{equation}\label{regularity:O-t}
\E \Big[ \sup_{t\in[0,T]} | \mathcal{O}_t |_{V}^p \Big]
 + \sup_{t\in[0,T]} \E \Big[ | \mathcal{O}_t |_{3}^p \Big] < \infty,
\end{equation}
and for $\alpha \in [0,3]$,
\begin{equation}
\| \mathcal{O}_t - \mathcal{O}_s \|_{L^p( \Omega , \dot{H}^{\alpha} )}
   \leq C |t-s|^{\text{min}\{ \frac 12,\frac{3-\alpha}{4} \}}.
\end{equation}
\end{lemma}

The following theorem states the well-posedness and spatio-temporal regularity of the mild solution for stochastic Cahn-Hilliard equation
\eqref{eq:CHC-abstract},  whose proofs can be found for example in \cite[Theorem 2.1 \& Theorem 2.2]{qi2020error}.
\begin{theorem}[Well-posedness and regularity of the mild solution]
	\label{thm:uniqueness-mild-solution}
	Under Assumptions \ref{assum:linear-operator-A}-\ref{assum:intial-value-data}, there is a unique mild solution of \eqref{eq:CHC-abstract} satisfying
	\begin{align}\label{eq:mild-solution}
		X(t)
		=
		E(t)X_0
		-
		\int_0^t \!\! E(t-s) A F(X(s))\,\mathrm{d} s
		+
		\int_0^t \!\! E(t-s)  \mathrm{d} W(s),
		\, t \in [0, T].
	\end{align}
	Furthermore, for  $p \geq 1$,
	\begin{align}\label{eq:spatial-regularity-mild-stoch}
		\sup_{t\in[0,T]}
		\|X(t)\|_{L^{p} ( \Omega, \dot{H}^3 ) }
		<
		\infty,
	\end{align}
	and for any $ \alpha \in [0,3]$,
	\begin{align}\label{eq:time-regularity-mild-stoch}
		\| X(t) - X(s) \|_{L^{p} ( \Omega, \dot{H}^{\alpha} ) }
		\leq
		C(t-s)^{\text{min}\{ \frac 12,\frac{3-\alpha}{4} \}},
\, 0 \leq s < t \leq T.
	\end{align}
\end{theorem}
Combining \eqref{eq:spatial-regularity-mild-stoch} and  \eqref{eq:norm-algebra} yields the following result.
\begin{corollary}
If Assumptions \ref{assum:linear-operator-A}-\ref{assum:intial-value-data} are valid,
then for all $p \ge 1$,
\begin{equation}\label{eq:spatial-regularity-mild-F(X)}
\sup_{t\in[0,T]}
\|F ( X(t) )\|_{L^{p} ( \Omega, \dot{H}^2 ) }  <  \infty.
	\end{equation}
\end{corollary}

\subsection{Introduction to Malliavin calculus}
A brief introduction to Malliavin calculus is given in this subsection.
For more details, one can consult the classical monograph \cite{Nualart:06}.
Define a Hilbert space $U_0=Q^{\frac12}(\dot{H})$ with inner product
$\langle u , v \rangle_{U_0}
= \langle Q^{-\frac12} u , Q^{-\frac12} v \rangle $.
Let
$\mathcal{G}:L^2 ([0,T],U_0 ) \rightarrow  L^2(\Omega)$
be an isonormal Gaussian process.
More precisely, for any deterministic mapping
$\phi \in L^2 ([0,T],U_0)$,
$\mathcal{G}(\phi)$
is centered Gaussian with the covariance structure
\begin{equation}
	\E \big[ \mathcal{G}(\phi_1) \mathcal{G}(\phi_2) \big]
	= \langle \phi_1 , \phi_2 \rangle_{L^2 ([0,T],U_0 )},
	\,\,
	\phi_1,\phi_2 \in L^2 ([0,T],U_0).
\end{equation}
For example (see e.g., \cite{andersson2016weak}), we define the cylindrical $Q$-Wiener process
\begin{equation}
W(t)u=\mathcal{G}(\chi_{[0,t]} \otimes u), \, u \in U_0, \, t \in [0,T].
\end{equation}
Given $u \in U_0$, the process $W(t)u, t \in [0, T]$, is a Brownian motion and we have
\begin{equation}
\E [W(t)u W(s)v] = \text{min}\{s,t\}
   \langle u,v \rangle_{U_0}, \,\,u,v \in U_0.
\end{equation}
Let $C_p^{\infty}(\R^M,\R)$ be the space of all $C^{\infty}$-mappings with polynomial growth.
We define the family of  smooth $\dot{H}$-valued cylindrical random variables as
\begin{equation}
	\mathcal{S}(H)= \Big\{ G = \sum_{i=1}^N g_i
	\big( \mathcal{G}(\phi_1), \ldots, \mathcal{G}(\phi_M) \big) h_i:
	 \phi_1,\cdots,\phi_M \in L^2 ([0,T],U_0),\,
g_i \in C_p^{\infty}(\R^M,\R),\,
 h_i \in \dot{H}, \, i \in \{ 1,\cdots,N\} \Big\}.
\end{equation}
 The Malliavin derivative of
$G \in \mathcal{S}(\dot{H})$ is an element of $\mathcal{L}_2(U_0,\dot{H})$ and given by
\begin{equation}
	\mathcal{D}_t G:=\sum_{i=1}^N \sum_{j=1}^M
	\partial_j g_i \big( \mathcal{G}(\phi_1),\ldots,\mathcal{G}(\phi_M)\big) h_i
	\otimes \phi_j(t),
\end{equation}
where $h_i \otimes \phi_j(t)$ denotes the tensor product, that is, for $1 \leq j \leq M$ and $1 \leq i \leq N$,
\begin{equation}
	\big(  h_i \otimes \phi_j(t) \big)(u)
	=  \langle  \phi_j(t),u \rangle_{U_0} h_i \in \dot{H},
	\quad  \forall\,\,u\in U_0,~h_i\in \dot{H},~ t\in[0,T].
\end{equation}
If $G$ is $\mathcal F_t$-measurable, then ${\mathcal D}_s G= 0$ for $s > t$.
The derivative operator ${\mathcal D}$ is known to be closable and we define $\mathbb D^{1,2}(\dot{H})$ as the closure of  $\mathcal{S}(\dot{H})$ with respect to the norm
\begin{equation}
	\|G\|_{\mathbb D^{1,2}(\dot{H})}
	=\Bigl( \E  \big[ \| G \|^2 \big] +
	\E\int_0^T \| {\mathcal D}_t G \|_{\mathcal L_2(U_0,\dot{H})}^2
	\dd t  \Bigr)^{\frac12}.
\end{equation}

We are now ready to give the Malliavin integration by parts formula.
For any $G \in \mathbb D^{1,2}(\dot{H})$ and  an adapted process
$\Psi \in L^2([0,T] \times \Omega,\mathcal{L}_2(U_0,\dot{H}))$,
\begin{equation}\label{Malliavin integration by parts}
	\E \left [ \left \langle
	\int_0^T \Psi(t) \dd W(t) , G \right \rangle \right ]
	= \E  \left [  \int_0^T \left \langle \Psi(t) ,
\mathcal{D}_t G \right \rangle_{\mathcal{L}_2(U_0,\dot{H})} \dd t  \right ],
\end{equation}
where the stochastic integral is  It\^o integral.
To simplify the notation, we often write $\mathcal{L}_2^0$ instead of $\mathcal{L}_2(U_0,\dot{H})$.
Next, we define
${\mathcal D}_s^u G=\langle {\mathcal D}_s G, u \rangle $
 the derivative in the direction $u \in U_0$.
Then the Malliavin derivative acting on the It\^o integral
$\int_0^t \Psi(r) \dd W(r)$ satisfies for all $u \in U_0$,
\begin{equation}\label{Malliavin derivative on adjoint}
	\mathcal D_s^u \int_0^t \Psi(r) \dd W(r)=
	\int_0^t \mathcal D_s^u  \Psi(r)\dd W(r)
	+ \Psi(s) u, \quad 0 \leq s \leq t \leq T.
\end{equation}
Given another separable Hilbert space $\mathcal H$,
if $\sigma \in C_b^1(\dot{H}, \mathcal{H})$ and
$ G \in \mathbb D^{1,2}(\dot{H})$,
then $\sigma(G) \in \mathbb D^{1,2}(\mathcal H)$ and
the chain rule
  holds as
${\mathcal D}_t (\sigma(G)) =
\sigma'(G)\mathcal{D}_t G$.

\section{Weak convergence rate of the spectral Galerkin method}
\label{sec:spatial-weak-analysis}

This section is devoted to the  weak error analysis of the spatial spectral Galerkin semi-discretization.
In the beginning,
we define a finite dimension space $ H_N = \mathrm{span} \{  e_1 , \cdots , e_N \}$ and the projection operator $P_N: \dot{H}^{\beta} \to H_N$ by $ P_N x = \sum_{i=1}^N \langle x , e_i \rangle e_i$
 for all $ x \in \dot{H}^{\beta}, \beta \in \R$.
As a result, $A$ commutes with $P_N$ and
\begin{equation}\label{estimate:P^N-I}
	\big\| \big( P_N - I \big) x \big\|
	\leq  C \lambda_N^{- \frac \beta 2} |x|_{\beta},
	\quad \forall ~ \beta \geq 0.
\end{equation}
Applying the spectral  Galerkin approximation to \eqref{eq:CHC-abstract} results in the finite-dimensional stochastic differential equation, given by
\begin{equation}\label{eq:spatial-discrete}
	\dd X^N(t) + A^2 X^N(t) \dd t  + A P_N F( X^N(t) ) \dd t
	=  P_N \dd W(t),
	\, t \in (0,T]; \, \, X^N(0) = P_N X_0,
\end{equation}
whose unique solution, in the mild form, is written as
\begin{equation}\label{eq:space-mild}
	X^N(t)= E(t) P_N  X_0
	-  \int_0^t E(t-s) A P_N F(X^N(s)) \dd s
	+  \int_0^t E(t-s) P_N \dd W(s).
\end{equation}
Similarly  to Lemma \ref{lem:stochastic-convolution}, the spatio-temporal regularity of the discrete stochastic convolution
$ \int_0^t E(t-s) P_N \dd W(s) $ ($ \mathcal{O}_t^N $ for short) (see e.g., \cite{qi2021existence}) enjoys
\begin{equation}
	\sup_{ N \in \N } \sup_{ t \in [0,T] }
 \E \Big[ | \mathcal{O}_t^N |_3^p \Big]
	< \infty, \, \forall p \ge 1,
\end{equation}
and for $\alpha \in [0,3]$,
\begin{equation}
	\sup_{ N \in \N }  \big \| \mathcal{O}_t^N - \mathcal{O}_s^N
	\|_{L^p( \Omega , \dot{H}^{\alpha} )}
    \leq C |t-s|^{ \text{min} \{ \frac 12 , \frac{3-\alpha}{4} \}},
\, \forall p \ge 1.
\end{equation}

It has to be noted that essential difficulties exist for analyzing a finite element method for the considered SPDE.
Indeed, the orthogonal projection $P_h$ can not commute with operator $A$, although $P_N$ commutes with $A$.
Moreover, compared with finite difference method, the spectral Galerkin method admits a simpler analysis,
whose approximated solution is smooth and allows better control of the
Lipschitz constant.
The proof of the following regularity results  is  given in
\cite[Lemma 3.4]{qi2021existence}.
\begin{proposition}[Spatio-temporal regularity of spatial semi-discretization] \label{prop:regularity-spatial-spectral}
	If Assumptions \ref{assum:linear-operator-A}-\ref{assum:intial-value-data} are satisfied, then the mild solution of the spatial approximation process
	\eqref{eq:space-mild} admits  for all $p \ge 1$,
	\begin{equation}\label{eq:spatial-L6-norm-mild-Galerkin}
		\sup_{N \in \N} \E \big[ \sup_{ t \in [0,T]}
		\|X^N(t)\|_{L^6(\mathbf{D},\R)}^p \big]
		< \infty.
	\end{equation}
	Moreover, we have
	\begin{equation}\label{eq:spatial-regularity-mild-Galerkin}
		\sup_{N \in \N}
		\sup_{ t \in [0,T]} \|X^N(t)\|_{L^{p} ( \Omega, \dot{H}^3 ) }
      < \infty,
	\end{equation}
	and  for any $ \alpha \in [0,3]$,
	\begin{equation}\label{eq:temporal-regularity-mild-Galerkin}
		\sup_{N \in \N}
		\| X^N(t) - X^N(s) \|_{L^{p} ( \Omega, \dot{H}^{\alpha} ) }
		\leq  C (t-s)^{ \text{min} \{ \frac 12 , \frac{3-\alpha}{4} \}},\,
	0 \leq s < t \leq T.
	\end{equation}
\end{proposition}
Combining \eqref{eq:spatial-regularity-mild-Galerkin} and \eqref{eq:norm-algebra} gives the next result.
\begin{corollary}
Under Assumptions
	\ref{assum:linear-operator-A}-\ref{assum:intial-value-data},
	\begin{equation}\label{eq:spatial-regularity-mild-F(XN)}
		\sup_{t\in[0,T]}
		\|F ( X^N(t) )\|_{L^{p} ( \Omega, \dot{H}^2 ) }  <  \infty,
         \, \forall p \ge 1.
	\end{equation}
\end{corollary}

The next result shows that $X^N(t)$ is differentiable in Malliavin sense.
\begin{proposition}[Boundedness of the Malliavin derivative]
	\label{Estimate of Malliavin derivative of the solution}
	Let Assumptions
	\ref{assum:linear-operator-A}-\ref{assum:intial-value-data} hold. Then the Malliavin derivative of $X^N(t)$ satisfies
\begin{align}\label{eq:malliavin-derivative}
\E\big[\| \mathcal{D}_s X^N(t)   \|_{\mathcal{L}_2(U_0,\dot{H})}^2\big] < \infty, \quad  0 \leq s \leq t \leq T.
\end{align}
\end{proposition}
\begin{proof}
The existence of the Malliavin derivative
$\mathcal{D}_s^y X^N(t)$ can be obtained by the standard argument such as the Picard iteration.
Here, we will focus on the bound \eqref{eq:malliavin-derivative}.
Taking the Malliavin derivative on the equation \eqref{eq:space-mild} in the direction $y \in U_0$ and using the chain rule yield that for
	$0 \leq s \leq t \leq T$,
	\begin{align}
		\mathcal{D}_s^y X^N(t)
		= E(t-s)P_N y - \int_s^t E(t-r)  A P_N F'( X^N(r) )
		\mathcal{D}_s^y X^N(r) \dd r.
	\end{align}
Following a standard strategy for the analysis of the Cahn--Hilliard equations, the proof of the upper bounds for $\mathcal{D}_s^y X^N(t)$ requires to exploit two energy estimates, in the $| \cdot |_{-1}$ and $| \cdot |$ norms.
First, observe that for all $t \ge s$, $\mathcal{D}_s^y X^N(t)$ is differentiable  and  satisfies
	\begin{align}\label{eq:time-differentiable}
		\frac {\dd \mathcal{D}_s^y X^N(t)}{\dd t}
		+  A^2 \mathcal{D}_s^y X^N(t)
		+  A P_N F'(X^N(t)) \mathcal{D}_s^y X^N(t) =0.
	\end{align}
	Multiplying $A^{-1} \mathcal{D}_s^y X^N(t)$ on both sides of the above equation  yields
	\begin{align}\label{eq:inner-product}
		\begin{split}
			\Big \langle \frac {\dd \mathcal{D}_s^y X^N(t)}{\dd t},
			 A^{-1} \mathcal{D}_s^y X^N(t) \Big \rangle
                + \langle A^2 \mathcal{D}_s^y X^N(t),
			A^{-1} \mathcal{D}_s^y X^N(t) \rangle
			+ \langle A P_N F'(X^N(t))\mathcal{D}_s^y X^N(t),
			A^{-1} \mathcal{D}_s^y X^N(t) \rangle=0.
		\end{split}
	\end{align}
	Next, integrating \eqref{eq:inner-product} over $[s , t]$ one obtains
	\begin{align}
		\begin{split}
			| \mathcal{D}_s^y X^N(t) |_{-1}^2 &= | y |_{-1}^2
		\!	- \! 2 \! \! \int_s^t \!\! \! |\mathcal{D}_s^y X^N(r)|_{1}^2 \dd r
		\!	- \! 2 \! \! \int_s^t \!\! \! \left \langle  F'(X^N(r))\mathcal{D}_s^y X^N(r),
			\mathcal{D}_s^y X^N(r) \right \rangle \dd r
			\\&
			= | y |_{-1}^2 \!
			- \! 2 \int_s^t \! \! \! |\mathcal{D}_s^y X^N(r)|_{1}^2 \dd r
			+ 2 \int_s^t \! \!\! \left \langle A^{\frac12} \mathcal{D}_s^y X^N(r),
			A^{-\frac12} \mathcal{D}_s^y X^N(r) \right \rangle \dd r
			\\&\quad
			- 6 \int_s^t \left \langle (X^N(r))^2 \mathcal{D}_s^y X^N(r),
			\mathcal{D}_s^y X^N(r) \right \rangle \dd r
			\\& \leq
			| y |_{-1}^2
			-  \int_s^t |\mathcal{D}_s^y X^N(r)|_{1}^2 \dd r
			+ \int_s^t |\mathcal{D}_s^y X^N(r)|_{-1}^2 \dd r,
		\end{split}
	\end{align}
where in the last step the elementary inequality $2ab \leq a^2+b^2$ was used.
	Hence, by  Gronwall's inequality we have
	\begin{equation}
		|\mathcal{D}_s^y X^N(t)|_{-1}^2 \leq C |y|_{-1}^2.
	\end{equation}
Therefore, one has
	\begin{equation}
		\int_s^t |\mathcal{D}_s^y X^N(r)|_1^2 \dd r \leq C |y|_{-1}^2.
	\end{equation}
Next, we may multiply by $\mathcal{D}_s^y X^N(t)$ both sides of \eqref{eq:time-differentiable} to get
	\begin{align}
		\begin{split}
			\left \langle \frac {\dd \mathcal{D}_s^y X^N(t)}{\dd t},
			\mathcal{D}_s^y X^N(t) \right\rangle
			+ \left \langle A^2 \mathcal{D}_s^y X^N(t),  \mathcal{D}_s^y X^N(t) \right \rangle
			 + \left \langle A P_N F'(X^N(t))\mathcal{D}_s^y X^N(t),
			\mathcal{D}_s^y X^N(t) \right \rangle=0.
		\end{split}
	\end{align}
Similarly, the  energy estimate in the  $| \cdot |$ norm is treated as follows:
\begin{align}
\begin{split}
| \mathcal{D}_s^y X^N(t) |^2 &= | y |^2
			- 2 \int_s^t \|A \mathcal{D}_s^y X^N(r)\|^2 \dd r
	- 2 \int_s^t \langle  F'(X^N(r))\mathcal{D}_s^y X^N(r),
	A \mathcal{D}_s^y X^N(r) \rangle \dd r
\\&\leq | y |^2
	- 2 \int_s^t \|A \mathcal{D}_s^y X^N(r)\|^2 \dd r
	+ 2\int_s^t \|A \mathcal{D}_s^y X^N(r)\|^2 \dd r
\\ &\quad
+  \tfrac12 \int_s^t \| F'(X^N(r)) \mathcal{D}_s^y X^N(r) \|^2 \dd r
\\ &   \leq |  y |^2
+ C \big( \sup_{r \in [s,t]} \| X^N(r) \|_{L^6}^4 +1 \big)
			\int_s^t |\mathcal{D}_s^y X^N(r)|_1^2 \dd r
\\ & \leq | y |^2
	+ C \, | y |_{-1}^2
	\big( \sup_{r \in [s,t]} \| X^N(r) \|_{L^6}^4 +1 \big)
\\ &  \leq | y |^2  + C \, | y |^2 \big( \sup_{r \in [s,t]} \| X^N(r) \|_{L^6}^4 +1 \big),
\end{split}
\end{align}
where in the first inequality the elementary inequality
$2ab \leq 2a^2 + \tfrac12 b^2$ was used.
What's more, H\"{o}lder's inequality $\| f g \| \leq C \| f \|_{L^3} \| g \|_{L^6}$ and Sobolev embedding inequality $\dot{H}^{\frac d3} \subset L^6, d=1,2,3$ were used in the above second inequality.
Choosing $y=Q^{\frac 12}e_i,i =\{1,2,\cdots\}$ and  taking expectation yield
\begin{equation}\label{eq:DSXNT}
\E\big[\| \mathcal{D}_s X^N(t)   \|_{\mathcal{L}_2(U_0,\dot{H})}^2\big]
 \leq C \, \|Q^{\frac 12}\|_{\mathcal{L}_2}^2
 \leq C \, \|A^{\frac 12} Q^{\frac 12}\|_{\mathcal{L}_2}^2< \infty,
\end{equation}
where \eqref{eq:ass-AQ-condition} and \eqref{eq:spatial-L6-norm-mild-Galerkin} were used.
\end{proof}

\begin{remark}
The trace-class noise (i.e., $\|Q^{\frac12}\|_{\mathcal{L}_2} < \infty$) is sufficient to obtain \eqref{eq:DSXNT} and thus Proposition \ref{Estimate of Malliavin derivative of the solution}.
\end{remark}

Let us now turn to some useful results on the nonlinear term $F$.
\begin{lemma}\label{lemma;F}
	Let $F$ be the Nemytskii operator defined in Assumption \ref{assum:nonlinearity},
	then for $d=1,2,3$,
	\begin{equation}\label{eq:F'-1}
		|F'(x)y|_1 \leq C \big( 1+|x|_{2}^2 \big) |y|_{1}, \,
       x \in \dot{H}^2 ,\, y \in \dot{H}^1,
	\end{equation}
	and
	\begin{align}\label{eq:F'--eta}
		\begin{split}
			|F'(\varsigma)\psi |_{-1}
			\leq C\big(1+|\varsigma|_2^2\big)
			|\psi|_{-1},
			\quad
			\forall
			\varsigma \in \dot{H}^{2},
			\psi \in \dot{H}.
		\end{split}
	\end{align}
\end{lemma}

\begin{proof}
The estimate \eqref{eq:F'-1} is an immediate consequence of
\cite[Lemma 3.2]{qi2021existence}.
To see \eqref{eq:F'--eta}, with the aid of the self-adjointness of $A^{-\frac12}$ and $F'(\varsigma)$, we have
\begin{equation}
\begin{split}
\| A^{-\frac12} F'(\varsigma)\psi \| &
  = \sup_{\|\xi\| \leq 1} \langle A^{-\frac12}   F'(\varsigma) \psi,
            \xi \rangle
  = \sup_{\|\xi\| \leq 1} \langle \psi,
                F'(\varsigma) A^{-\frac12} \xi \rangle
  =\sup_{\|\xi\| \leq 1} \langle A^{-\frac12} \psi,
              A^{\frac12}  F'(\varsigma) A^{-\frac12} \xi \rangle
  \\ & \leq | \psi |_{-1} \sup_{\|\xi\| \leq 1}
                   | F'(\varsigma) A^{-\frac12} \xi |_1
       \leq C\big(1+|\varsigma|_2^2\big)|\psi|_{-1}.
\end{split}
\end{equation}
The condition $\psi \in \dot{H}$ is used in the above first and second identities.
This finishes the proof.
\end{proof}

Now, we are well prepared to carry out the weak error analysis of the spatial semi-discretization.
\begin{theorem}[Weak convergence rate of the spatial approximation]
	\label{thm:weak-space}
Let $X(T)$ and $X^N(T)$, given by \eqref{eq:mild-solution} and \eqref{eq:space-mild}, be the solution of problems \eqref{eq:CHC-abstract} and  \eqref{eq:spatial-discrete} respectively.
	Let  Assumptions \ref{assum:linear-operator-A}-\ref{assum:intial-value-data} hold.
Then   for $\Phi \in C_b^2$,
	there exists a constant $C > 0$ such that
\begin{equation}
		\big| \E [ \Phi(X(T))] - \E [\Phi(X^N(T))] \big|
		\leq C \, \lambda_N^{-2}.
	\end{equation}
\end{theorem}

\begin{proof}
	By introducing two processes
	$\bar{X}(t)= X(t) - \mathcal{O}_t$
	and
	$\bar{X}^N(t) = X^N(t) - \mathcal{O}^N_t$,
	we can separate the error
	$\E\big[\Phi(X(T))\big]
	-\E \big[\Phi(X^N(T))\big]$ into two terms as follows
	\begin{equation}\label{full:separation}
		\begin{split}
			\E\big[\Phi(X(T))\big] -\E \big[\Phi(X^N(T))\big]
			 &=
			\Big(\E\big[\Phi(\bar{X}(T)+\mathcal{O}_T)\big]
			-\E \big[\Phi(\bar{X}^N(T)+\mathcal{O}_T)\big]\Big) \\
			& \quad +
			\Big(\E\big[\Phi(\bar{X}^N(T)+\mathcal{O}_T)\big]
			-\E \big[\Phi(\bar{X}^N(T)+\mathcal{O}_T^N)\big]\Big)
			\\
			&=:I_1+I_2.
		\end{split}
	\end{equation}
	To estimate $I_1$, it suffices to consider the strong convergence between $\bar{X}(T)$ and $\bar{X}^N(T)$.
	To be specific, by the Taylor expansion and triangle inequality we have
	\begin{equation}\label{eq:I1}
		\begin{split}
			|I_1| &= \Big|\E\big[\Phi(\bar{X}(T)+\mathcal{O}_T)\big]
			-\E \big[\Phi(\bar{X}^N(T)+\mathcal{O}_T)\big]\Big|
			\leq
			C \big| \E [
               \| \bar{X}(T) - \bar{X}^N(T) \| ]  \big|
			\\& \leq  C
			\| \bar{X}(T) - P_N \bar{X}(T) \|_{L^2(\Omega,\dot{H})}
			+   C
			\| P_N \bar{X}(T) - \bar{X}^N(T) \|_{L^2(\Omega,\dot{H})}.
		\end{split}
	\end{equation}
To bound the first error term
$\| \bar{X}(T) - P_N \bar{X}(T) \|_{L^2(\Omega,\dot{H})}$,
we need an estimate on $\bar{X}(t),t \in [0,T]$, that is,
\begin{equation}\label{regular:X-bar}
\begin{split}
 \|  \bar{X}(t)  &
 \|_{L^p(\Omega,\dot{H}^4)}
	 = \| A^{2}  \bar{X}(t)\|_{L^p(\Omega,\dot{H})}
	\\ & \leq \| A^{2 } E(t) X_0 \|_{L^p(\Omega,\dot{H})}
		+ \left\| \int_0^t A^{2} E( t - s )
			A P F(X(t))\,\dd s \right\|_{L^p(\Omega,\dot{H})}
    \\	& \quad + \left\| \int_0^t A^{2} E( t - s )
			A P   (F( X(t)) -  F (X(s) ))\,\dd s
                  \right\|_{L^p(\Omega,\dot{H})}
	\\&  \leq C \Big( | X_0 |_4
          +  \|F(X(t))\|_{L^p(\Omega,\dot{H}^2)}
       + \int_0^t (t-s)^{-1}
       \| P (F( X(t)) -  F (X(s) )) \|_{L^p(\Omega,\dot{H}^2)}
         \dd s \Big)
    \\&  \leq C \Big( | X_0 |_4
          +  \|F(X(t))\|_{L^p(\Omega,\dot{H}^2)} \Big)
    \\&   \quad + C \Big( 1 + \sup_{r\in [0,t]} \|X(r)\|_{L^{4p}(\Omega,\dot{H}^2)}^2 \Big) \cdot
     \int_0^t (t-s)^{-1}
       \|  X(t) -  X(s)  \|_{L^{2p}(\Omega,\dot{H}^2)}   \dd s
    \\& \leq C \Big( 1
       + \int_0^t (t-s)^{-1} \cdot (t-s)^{\frac 14} \dd s \Big)
    \\&  <  \infty,
\end{split}
\end{equation}
where \eqref{I-spatial-temporal-S(t)} and \eqref{IV-spatial-temporal-S(t)} were used in the above second inequality, \eqref{eq:norm-algebra} was used in the above third inequality, \eqref{eq:spatial-regularity-mild-stoch}-\eqref{eq:spatial-regularity-mild-F(X)} were used in the above fourth inequality.
As a result, by using \eqref{estimate:P^N-I}, we  get
	\begin{equation}\label{eq:I-PN-XT}
		\| \bar{X}(T) - P_N \bar{X}(T) \|_{L^p(\Omega,\dot{H})}
		=
		\| (I- P_N) A^{-2 }
		A^{2} \bar{X}(T) \|_{L^p(\Omega,\dot{H})}
		\leq
		C \lambda_N^{-2}.
	\end{equation}
In the next step, we consider the second term of \eqref{eq:I1} in  the treatment of $I_1$.
For convenience,  we denote
$ e^N(t)=P_N \bar{X}(t) - \bar{X}^N(t)$, which satisfies
	\begin{equation}
		\tfrac{\dd}{\dd t}  e^N(t)
		+ A^2  e^N(t) + A P_N
		\big[ F(\bar{X}(t)+\mathcal{O}_t)
		- F(\bar{X}^N(t) + \mathcal{O}^N_t) \big] =0.
	\end{equation}
We multiply the above identity by $A^{-1}{e^N(t)}$ to get
	\begin{equation}\label{eq:space-frechet-estimate}
		\begin{split}
			\tfrac12 \tfrac{\dd}{\dd  t}  | e^N(t)|_{-1}^2
			+ |  e^N(t) |_1^2
			& =
			- \langle  e^N(t),
			F(\bar{X}( t)+\mathcal{O}_{t})-
               F(P_N\bar{X}( t)+\mathcal{O}_{ t})
                 \rangle
			\\& \quad - \langle {e^N(t)},
			F(P_N\bar{X}( t)+\mathcal{O}_{ t})
- F(\bar{X}^N(t)+\mathcal{O}_{t})
                 \rangle
			\\& \quad - \langle {e^N(t)},
			F(\bar{X}^N({t})+\mathcal{O}_{t})
   - F(\bar{X}^N(t)+\mathcal{O}^N_{t})
                 \rangle
         \\&  \leq
			\tfrac12 \| e^N(t)\|^2 + \tfrac12
			\|F(\bar{X}( t)+\mathcal{O}_{t})
- F(P_N\bar{X}(t)+\mathcal{O}_{ t})\|^2
			\\& \quad  + \| e^N(t)\|^2 + | e^N(t)|_1 \cdot
			|F(\bar{X}^N(t)+\mathcal{O}_{ t})-
                    F(\bar{X}^N(t)+\mathcal{O}^N_{t})|_{-1}
			\\& \leq
			\tfrac32 \|e^N(t)\|^2 + \tfrac12
			\|F(\bar{X}(t)+\mathcal{O}_{t})
   - F(P_N\bar{X}(t)+\mathcal{O}_{t})\|^2
			\\& \quad + \tfrac14 |e^N(t)|_1^2 +
			|F(\bar{X}^N(t)+\mathcal{O}_{t})-
                    F(\bar{X}^N( t)+\mathcal{O}^N_{ t}|_{-1}^2
			\\&\leq \tfrac34 |e^N(t)|_1^2 + \tfrac98 | e^N(t)|_{-1}^2
			+ C \|\bar{X}(t) - P_N \bar{X}( t)\|^2
			(1+ |\bar{X}(t)|_2^4 + |\mathcal{O}_{t}|_2^4)
			\\& \quad + C |\mathcal{O}_{t} - \mathcal{O}^N_{ t}|_{-1}^2
			(1+ |\bar{X}^N({t})|_2^4 + |\mathcal{O}_{ t}|_2^4),
		\end{split}
	\end{equation}
    where in the above first inequality we used Young's inequality $ab\leq \frac12 a^2 + \frac12 b^2$, (21) and Cauchy-Schwartz inequality.
    Also, \eqref{eq:local-condition}, Sobolev embedding inequality $\dot{H}^2 \subset V$, Young's inequality $\frac32 ab \leq \frac12 a^2 + \frac98 b^2$, Taylor's expansion and  Lemma \ref{lemma;F} were used in the above last inequality.
	By Gronwall's inequality, we further deduce
	\begin{equation*}
		\begin{split}
		|e^N(T)|_{-1}^2 + \int_0^T |  e^N(t) |_1^2 \dd t
			& \leq  C \int_0^T \|\bar{X}( t) - P_N \bar{X}( t)\|^2
			(1+ |\bar{X}(t)|_2^4 + |\mathcal{O}_{t}|_2^4) \dd t
			\\& \quad + C \int_0^T
			|\mathcal{O}_{t} - \mathcal{O}^N_{t}|_{-1}^2
			(1+ |\bar{X}^N(t)|_2^4 + |\mathcal{O}_{t}|_2^4) \dd t.
		\end{split}
	\end{equation*}
Applying \eqref{regularity:O-t} and \eqref{estimate:P^N-I} gives
\begin{equation}\label{eq:O-ON}
		\|\mathcal{O}_{ t} - \mathcal{O}^N_{t}
                       \|_{L^{p}(\Omega,\dot{H}^{-1})}
		= \|(I-P_N) A^{-2} A^{\frac{3}2} \mathcal{O}_{\color{red} t}\|_{L^{p}(\Omega,\dot{H})}
		\leq C \lambda_N^{-2}
   ~\| \mathcal{O}_t \|_{L^{p}(\Omega,\dot{H}^{3})}
   \leq C \lambda_N^{-2}.
\end{equation}
With the aid of the regularity of $X(T)$ and $X^N(T)$, \eqref{eq:I-PN-XT},
H\"{o}lder's inequality and \eqref{eq:O-ON},
one can find that
	\begin{equation}\label{eq:int-e-1}
		\begin{split}
			\Big\| \int_0^T |  e^N(t) |_1^2 \dd t \Big\|_{L^p(\Omega,\R)}
			& \leq C \int_0^T
	\|\bar{X}(t) - P_N \bar{X}( t)\|_{L^{4p}(\Omega,\dot{H})}^2 \dd  t
 + C \int_0^T \|\mathcal{O}_{t}
    -\mathcal{O}^N_t \|_{L^{4p}(\Omega,\dot{H}^{-1})}^2 \dd  t
			\\&    \leq C \lambda_N^{-4}.
		\end{split}
	\end{equation}
	We  are now ready to estimate
\begin{equation}\label{eq:e(T)}
\begin{split}
\|  e^N(T) \|_{L^p(\Omega,\dot{H})} & =
\Big\| P_N \Big( E(T) X_0 - \int_0^T E(T-s) A F(X(s)) \dd s \Big)
\\& \quad  - \Big( E(T)P_N X_0 - \int_0^T E(T-s) A P_N F(X^N(s)) \dd s \Big) \Big\|_{L^p(\Omega,\dot{H})}
\\& = \left\| \int_0^T E(T-s)  A P_N  (F(X(s))-F(X^N(s))) \dd s \right\|_{L^p(\Omega,\dot{H})}
\\& \leq
\left\| \int_0^T E(T-s) A P_N (F(\bar{X}(s)+\mathcal{O}_s)- F(P_N\bar{X}(s)+\mathcal{O}_s)) \dd s \right\|_{L^p(\Omega,\dot{H})}
\\& \quad + \left\| \int_0^T  E(T-s) A P_N ( F(P_N\bar{X}(s)+\mathcal{O}_s)
			-F(\bar{X}^N(s)+\mathcal{O}_s)) \dd s \right\|_{L^p(\Omega,\dot{H})}
\\& \quad + \left\| \int_0^T E(T-s)A  P_N ( F(\bar{X}^N(s)+\mathcal{O}_s)-
			F(\bar{X}^N(s)+\mathcal{O}^N_s)) \dd s \right\|_{L^p(\Omega,\dot{H})}
			\\&=:  e^N_1(T) + e^N_2(T) +  e^N_3(T).
\end{split}
\end{equation}
Again, by \eqref{I-spatial-temporal-S(t)}, \eqref{eq:local-condition}, \eqref{regularity:O-t},\eqref{estimate:P^N-I},  \eqref{regular:X-bar} and Sobolev embedding inequality  $\dot{H}^2 \subset V$, we have
	\begin{equation}\label{eq:e1(T)}
		\begin{split}
			e^N_1(T) &=
			\left\| \int_0^T E(T-s) A P_N (F(\bar{X}(s)+\mathcal{O}_s)- F(P_N\bar{X}(s)+\mathcal{O}_s)) \dd s \right\|_{L^p(\Omega,\dot{H})}
			\\& \leq
			C \int_0^T (T-s)^{-\frac12}
			\|\bar{X}(s) - P_N \bar{X}(s)\|_{L^{2p}(\Omega,\dot{H})}  \dd s \times
			\big(1+ \sup_{s\in[0,T]}\|\bar{X}(s)\|_{L^{4p}(\Omega,\dot{H}^2)}^2
			+ \sup_{s\in[0,T]} \|\mathcal{O}_s\|_{L^{4p}(\Omega,\dot{H}^2)}^2\big)
			\\& \leq
			C \lambda_N^{-2}.
		\end{split}
	\end{equation}
	From \eqref{I-spatial-temporal-S(t)},
	\eqref{eq:F'-1} in Lemma \ref{lemma;F}, H\"{o}lder's inequality,
	\eqref{eq:int-e-1} and regularity of $X(t)$ and $X^N(t)$,
	it follows that
	\begin{equation}
		\begin{split}
			e^N_2(T) & \leq C
			\left\|
			\int_0^T (T-s)^{{  -\frac14}}
			\big|F(P_N\bar{X}(s)+\mathcal{O}_s)
			-F(\bar{X}^N(s)+\mathcal{O}_s) \big|_1 \dd s
			\right\|_{L^p(\Omega,\R)}
			\\& \leq C
			\left\|
			\int_0^T (T-s)^{-\frac14}  | e^N(s)|_1
			\big( 1 + |\bar{X}(s)|_2^2 + |\bar{X}^N(s)|_2^2
			+ |\mathcal{O}_s|_2^2  \big) \dd s
			\right\|_{L^p(\Omega,\R)}
			\\&\leq C
			\left\|
			\int_0^T | e^N(s)|_1^2 \dd s
			\right\|_{L^p(\Omega,\R)}^{\frac12}
		\left(
			\int_0^T (T-s)^{-\frac12} \dd s \right)^{\frac12}
			\\& \leq C \lambda_N^{-2}.
		\end{split}
	\end{equation}
	Similarly  to the estimate of \eqref{eq:e1(T)}  with
	\eqref{eq:F'--eta} and \eqref{eq:O-ON} instead,
	we obtain
	\begin{equation}
		\begin{split}
			 e^N_3(T) &=
			\left\| \int_0^T E(T-s)  A^{\frac32}
			A^{-\frac12}  P_N
			(F(\bar{X}^N(s)+\mathcal{O}_s)- F(\bar{X}^N(s)+\mathcal{O}^N_s))
			\dd s \right\|_{L^p(\Omega,\dot{H})}
			\\& \leq
			C \int_0^T (T-s)^{-\frac34}
			\|\mathcal{O}_s - \mathcal{O}^N_s
			\|_{L^{2p}(\Omega,\dot{H}^{-1})}  \dd s
			\big(1+ \sup_{s\in[0,T]}\|\bar{X}^N(s)\|_{L^{4p}(\Omega,\dot{H}^2)}^2
			+ \sup_{s\in[0,T]} \|\mathcal{O}_s\|_{L^{4p}(\Omega,\dot{H}^2)}^2\big)
			\\& \leq
			C \lambda_N^{-2}.
		\end{split}
	\end{equation}
Therefore,	gathering estimates of  $ e^N_1(T)$, $ e^N_2(T)$ and
$ e^N_3(T)$ together yields
\begin{equation}\label{eq:e(T)-order}
\| P_N \bar{X}(T) - \bar{X}^N(T) \|_{L^2(\Omega,\dot{H})} \leq
C \lambda_N^{-2}.
\end{equation}
Combining it with \eqref{eq:I-PN-XT}  yields
\begin{equation}\label{space1}
\| \bar{X}(T) - \bar{X}^N(T) \|_{L^2(\Omega,\dot{H})} \leq
C \lambda_N^{-2},
\end{equation}
and thus $|I_1|\le C \lambda_N^{-2}$.
Next, we turn to the estimate of $|I_2|$.
Using Taylor's expansion and the triangle inequality, we get
\begin{align}\label{full;I_2}
\begin{split}
|I_2| &= \Big| \E \Big[ \Phi^{'} (X^N(T))(\mathcal{O}_T-\mathcal{O}_T^N) +\int_0^1 \Phi^{''}(X^N(T)+\lambda(\mathcal{O}_T-\mathcal{O}_T^N)) (\mathcal{O}_T-\mathcal{O}_T^N,\mathcal{O}_T-\mathcal{O}_T^N)
	(1-\lambda)\dd \lambda\Big]\Big|
	\\&\leq \Big|\E\big[\Phi' (X^N(T))(I-P_N)\mathcal{O}_T
			\big]\Big|
	+C \,\E\big[\| \mathcal{O}_T-\mathcal{O}^N_T\|^2\big].
		\end{split}
	\end{align}
The second term can be easily bounded by utilizing  \eqref{estimate:P^N-I}  and the moment bound for $|\mathcal{O}_T|_3$ in Lemma \ref{lem:stochastic-convolution}, that is
\begin{align}\label{space2}
\E \Big[ \Big\| \mathcal{O}_T - \mathcal{O}^N_T \Big\|^2 \Big]
	= \E \Big[ \Big\| (I-P_N) \mathcal{O}_T \Big\|^2 \Big]
		\leq C \lambda_N^{-3}.
\end{align}
For the first term,  \eqref{eq:malliavin-derivative} in  Proposition \ref{Estimate of Malliavin derivative of the solution},
the Malliavin integration by parts formula
\eqref{Malliavin integration by parts}, the chain rule of the Malliavin derivative, \eqref{I-spatial-temporal-S(t)}, \eqref{estimate:P^N-I} and \eqref{eq:ass-AQ-condition} enable us to obtain
\begin{align}\label{eq:space-rate-Malliavin}
\begin{split}
  \Big|  \E \big[ \Phi' ( X^N(T) )  ( I - P_N )  \mathcal{O}_T \big] \Big|
    &  = \Big| \E  \Big[ \Big \langle \int_0^T ( I - P_N ) E(T-s) \dd W(s), \Phi^{'} (X^N(T)) \Big \rangle \Big]  \Big|
	\\ & = \Big| \E \int_0^T \left< (I-P_N)E(T-s),
       \mathcal{D}_s \Phi^{'} (X^N(T)) \right>_{\mathcal{L}_2^0}\dd s \Big|
	\\ & \leq C \, \E\int_0^T \big\| ( I - P_N ) E(T-s)
			\big\|_{\mathcal{L}_2^0}
			\| \Phi^{''} (X^N(T)) \|_{\mathcal{L}}
			\| \mathcal{D}_s X^N(T) \|_{\mathcal{L}_2^0} \, \dd s
	\\&  \leq C \, \int_0^T \|(I-P_N) E(T-s) A^{-\frac12}\|_{\mathcal{L}}
			\|A^{\frac12}Q^{\frac{1}{2}}\|_{\mathcal{L}_2} \dd s
	\\&  \leq C \, \lambda_N^{- 2} \int_0^T (T-s)^{-\frac34} \dd s
         \leq C \, \lambda_N^{- 2}.
\end{split}
\end{align}
Hence, we obtain $|I_2|\le C \lambda_N^{-2}$.
Gathering it with $|I_1|\le C \lambda_N^{-2}$ then concludes the proof.
\end{proof}

\section{Weak convergence rate of the backward Euler method}
\label{sec:full-weak-analysis}

Based on the spatial spectral Galerkin approximation \eqref{eq:spatial-discrete},
this section concerns  the weak error analysis of a backward Euler method in the temporal direction.
We divide the interval $[0,T]$ into $M$ equidistant subintervals  with the time step-size $\tau=\tfrac{T}{M}$ and denote the nodes $t_m=m\tau$ for
$m \in \{0, 1,\cdots,M \},\,  M \in \mathbb{N}^{+}$.
Then,  the fully discrete scheme reads
\begin{equation}\label{eq:full-discrete-method-1}
	X_{t_m}^{M,N} - X_{t_{m-1}}^{M,N} + \tau A^2 X_{t_m}^{M,N}
	+ \tau P_N A F(X_{t_m}^{M,N})  = P_N \Delta W_m,
	\quad X_0^{M,N}= P_N X_0,
\end{equation}
where $\Delta W_m :=W(t_m)-W(t_{m-1})$ for short.
By introducing a family of operators $\{E_{\tau,N}^m\}_{m=1}^M$:
$
	E_{\tau,N}^m v = (I + \tau A^2)^{-m} P_N v
     = \sum_{j=1}^N ( 1 + \tau \lambda_j^2)^{-m}
     \langle v,e_j \rangle e_j
$, $\forall\  v \in \dot{H}$,
we have
\begin{equation}\label{eq:full-discrete-method-2}
	X_{t_m}^{M,N} = E_{\tau,N}^m  X_0
	- \tau \sum_{j=1}^{m} E_{\tau,N}^{m-j+1} A  F(X_{t_j}^{M,N})
	+ \sum_{j=1}^{m} E_{\tau,N}^{m-j+1}  \Delta W_j.
\end{equation}
Thanks to \cite[Theorem C.2]{stuart1996dynamical}, the implicit scheme \eqref{eq:full-discrete-method-1} is well-defined.
More details can be found in \cite{qi2021existence}.
Following the proof of \cite[(2.10)]{furihata2018strong}, it is easy to check that the operator $E_{\tau,N}^m$ satisfies
\begin{equation}\label{eq:E-tau-N-estimate}
	\| A^{\mu} E_{\tau,N}^m v \| \leq C t_m^{-\frac \mu 2} \| v \|,
	\quad \mu \in [0,2],\, v \in \dot{H},\, m \in \{1,2,\cdots,M\}
\end{equation}
and there exists a constant $C$ such that for all $v \in \dot{H}$,
\begin{equation}\label{eq:E-tau-N-2-estimate}
\Big( \tau \sum_{j=1}^m \| A E_{\tau,N}^j v \|^2 \Big)^{\frac12}
   \leq C \|v\|.
\end{equation}
The regularity of the fully discrete approximation is derived in the following result.
\begin{proposition}\label{prop:full-regularity}
Let Assumptions \ref{assum:linear-operator-A}-\ref{assum:intial-value-data} be satisfied,
then we have for all $ p \geq 1$,
\begin{equation}\label{eq:moment-bounds-full}
\sup_{N \in \N}\sup_{m \in \{0,1,\cdots,M\} }
  \|X^{M,N}_{t_{m}}\|_{L^p(\Omega,\dot{H}^2)}  < \infty.
\end{equation}
\end{proposition}
\begin{proof}
Firstly, by the proof in \cite[Theorem 4.1]{qi2021existence}, we have for $\eta \in (\frac32,2)$ and all $ p \geq 1$,
\begin{equation}
\sup_{N \in \N}\sup_{m \in \{0,1,\cdots,M\} }
\|X^{M,N}_{t_{m}}\|_{L^p(\Omega,\dot{H}^{\eta})} < \infty.
\end{equation}
Next, from \eqref{eq:E-tau-N-estimate}, the Burkholder-Davis-Gundy-type inequality, \eqref{eq:E-tau-N-2-estimate}, \eqref{eq:ass-AQ-condition}, \eqref{eq:X0} and Sobolev embedding inequality
$\dot{H}^{\eta} \subset V$, it follows  that
\begin{equation}
\begin{split}
\sup_{N \in \N} & \sup_{m \in \{0,1,\cdots,M\} }
  \|X^{M,N}_{t_{m}}\|_{L^p(\Omega,\dot{H}^2)}
  \\ &  \leq  \|X_0\|_{L^p(\Omega,\dot{H}^2)}
     + \tau \sup_{N \in \N}\sup_{m \in \{0,1,\cdots,M\} }
        \sum_{j=1}^m t_{m-j+1}^{-\frac34}
        \|P F (X_{t_j}^{M,N})\|_{L^p(\Omega,\dot{H}^1)}
  \\ & \quad + \sup_{N \in \N}\sup_{m \in \{0,1,\cdots,M\} }
     \Big( \tau \sum_{j=1}^m
        \big\| A^{\frac32} E_{\tau,N}^{m-j+1} Q^{\frac12}
           \big\|_{\mathcal{L}_2}^2 \Big)^{\frac12}
  \\ & \leq C \Big(
       1 + \|A^{\frac12}Q^{\frac12}\|_{\mathcal{L}_2}^2
      +  \tau \sup_{N \in \N}\sup_{m \in \{0,1,\cdots,M\} }
      \sum_{j=1}^{m} t_{m-j+1}^{-\frac34}
      \Big\| |X_{t_j}^{M,N}|_1 ~ \|X_{t_j}^{M,N}\|_V^2
        \Big\|_{L^p(\Omega,\R)}\Big)
  \\ & \leq C \Big( 1 + \sup_{m \in \{0,1,\cdots,M\} }
     \tau   \sum_{j=1}^{m} t_{m-j+1}^{-\frac34}
    \sup_{N \in \N}\sup_{j \in \{0,1,\cdots,M\} }
    \|X_{t_j}^{M,N}\|_{L^{3p}(\Omega,\dot{H}^{\eta})} \Big)
  \\ & < \infty.
\end{split}
\end{equation}
This completes the proof.
\end{proof}

Before presenting the main theorem,  we introduce the notation
$\lfloor s \rfloor:=\max\{0,\tau,\cdots,m\tau,\cdots \}
	\cap [0,s] $,
$	\lceil s \rceil:=\min\{0,\tau,\cdots,m\tau,\cdots \}
	\cap [s,T] $
and $[s]:=\frac {\lfloor s \rfloor} {\tau}$.
The fully discrete approximation operator is then defined by
\begin{equation}
	\Psi_{\tau}^{M,N} (t) := E(t) P_N - E_{\tau,N}^k,
	\quad t \in [t_{k-1} , t_k ), \quad k \in \{1,2,\cdots,M\}.
\end{equation}
The following lemma of the fully discrete approximation operator plays a pivotal role in the weak convergence analysis.
\begin{lemma}\label{lem;estimate-fully-approximation-operator}
	Under Assumption \ref{assum:linear-operator-A},
	we have the following statements.
	\begin{enumerate}
		\item[(i)] Let $\rho \in [0,4]$,
              there exists a constant $C$ such that for $t > 0$,
		\begin{equation}
			\| \Psi_{\tau}^{M,N} (t) u \|
			\leq C \,  t^{-\frac{\rho}{4}} \, |u|_{- \rho}, \,
      u \in \dot{H}^{-\rho}.
		\end{equation}
		
        \item[(ii)] Let $\beta \in [0,4]$, there exists a constant $C$ such that for $t > 0$,
		\begin{equation}
			\| \Psi_{\tau}^{M,N} (t) u \|
			\leq C \,  \tau^{\frac{\beta}{4}} \, |u|_{\beta},
          \, u \in \dot{H}^{\beta}.
		\end{equation}
		
          \item[(iii)] Let $\alpha\in [0,4]$, there exists a constant $C$ such that for $t > 0$,
		\begin{equation}
			\| \Psi_{\tau}^{M,N} (t) u \|
			\leq C \, \tau^{\frac{4-\alpha} 4} \, t^{-1}
       \, |u|_{- \alpha}, \, u \in \dot{H}^{-\alpha}.
		\end{equation}
		
         \item[(iv)] Let $\mu \in [0,4]$, there exists a constant $C$ such
           that for $t > 0$,
		\begin{equation}
			\| \Psi_{\tau}^{M,N} (t) u \|
			\leq C \, \tau \cdot t^{-\frac{4-\mu} 4} \,
			|u|_{\mu }, \, u \in \dot{H}^{\mu }.
		\end{equation}
		
         \item[(v)] Let $\nu \in [0,4]$, there exists a constant $C$ such that for $t > 0$,
		\begin{equation}
			\Big( \int_0^t
			\| \Psi_{\tau}^{M,N} (s) u \|^2 \mathrm{d} s \Big)^{\frac12}
			\leq C \, \tau^{\frac{\nu} 4}
			|u|_{\nu - 2},\, u \in \dot{H}^{\nu - 2}.
		\end{equation}

        \item[(vi)] Let $\delta\in [0,4]$, there exists a constant $C$ such that for $t > 0$,
		\begin{equation}
			\Big\| \int_0^t \Psi_{\tau}^{M,N} (s) u \mathrm{d} s \Big\|
			\leq C \, \tau^{\frac{4-\delta} 4}
			|u|_{- \delta},\, u \in \dot{H}^{ - \delta}.
		\end{equation}
	\end{enumerate}
\end{lemma}
\begin{proof}
	Elementary fact in \cite[Lemma 5.3]{qi2021existence} yields   (i), (ii), (iii), (v) and (vi).
	We then use the standard interpolation argument to prove   (iv).
	For $\mu=0$, it is a consequence of (iii) with $\alpha=0$ and for $\mu=4$,
    it is a consequence of  (ii) with $\beta=4$.
\end{proof}

For clarity of exposition, we denote
$\mathcal{O}_T^{M,N}:=\sum_{j=1}^{M} E_{\tau,N}^{M-j+1} \Delta W_j =\int_0^T E_{\tau,N}^{M-[s]} \dd W(s).$
The next lemma gives the estimate between $\mathcal{O}_{t_m}^N$ and $\mathcal{O}_{t_m}^{M,N}$.
\begin{lemma}\label{lem:ON-OMN}
	Under Assumptions \ref{assum:linear-operator-A} and \ref{assum:noise-term},  we have
	for $p \geq 1$,
	\begin{equation}
		\sup_{m \in \{1,2,\cdots,M\}}
		\big\| \mathcal{O}_{t_m}^{M,N}
		- \mathcal{O}_{t_m}^{N} \big\|_{L^p(\Omega,\dot{H}^{-\beta})}
		\leq C \, \tau^{\frac{3+\beta}4}, \, \beta \in [-3,1].
	\end{equation}
\end{lemma}
\begin{proof}
	The Burkholder-Davis-Gundy  inequality and (v) in Lemma \ref{lem;estimate-fully-approximation-operator} with
	$\nu=3+\beta$  yield
	\begin{equation}
		\begin{split}
			\big\| \mathcal{O}_{t_m}^{M,N}
		- \mathcal{O}_{t_m}^{N} \big\|_{L^p(\Omega,\dot{H}^{-\beta})}
			& \leq C ~\Big(
			 \int_{0}^{t_m}
			\| \Psi_{\tau}^{M,N} (t_m -s)
			A^{-\frac{\beta}2} Q^{\frac12} \|_{\mathcal{L}_2}^2 \dd s
			\Big)^{\frac12}
			\\& \leq C \, \tau^{\frac{3+\beta}4}
			\|A^{\frac12} Q^{\frac12} \|_{\mathcal{L}_2}
			 \leq C \, \tau^{\frac{3+\beta}4}.
		\end{split}
	\end{equation}
	This finishes the proof.
\end{proof}

The following theorem shows the weak convergence rate of the
temporal semi-discretization.
\begin{theorem}[Weak convergence rate of the temporal approximation]
\label{th:order-time}
	Suppose Assumptions \ref{assum:linear-operator-A}-\ref{assum:intial-value-data} are satisfied.
	Let $X^N(T)$ and $X_T^{M,N}$ be given by \eqref{eq:space-mild} and \eqref{eq:full-discrete-method-2}, respectively.
	Then, we have for $\Phi \in C_b^2$,
	\begin{equation}
		\big| \E [\Phi(X^N(T))] -  \E [\Phi(X_T^{M,N})] \big|
		\leq C \, \tau.
	\end{equation}
\end{theorem}

\begin{proof}
At first, we define
$\bar{X}_T^{M,N}= {X}_T^{M,N} - \mathcal{O}_T^{M,N}$
and separate the above error into
\begin{align}
\begin{split}
\E \Big[ \Phi (X^N(T)) \Big] - \E \Big[ \Phi (X_T^{M,N}) \Big]
	& =
 \Big( \E \big[ \Phi (\bar{X}^N(T) + \mathcal{O}_T^{M,N}) \big]- \E \big[ \Phi (\bar{X}_T^{M,N} + \mathcal{O}_T^{M,N}) \big] \Big)
	\\& \quad  +
    \Big( \E \big[ \Phi (\bar{X}^N(T) + \mathcal{O}_T^N) \big]
	- \E \big[ \Phi (\bar{X}^N(T) + \mathcal{O}_T^{M,N}) \big] \Big)
	\\&  =: K_1  + K_2.
\end{split}
\end{align}
To estimate $K_1$, it suffices to bound
$\big\| \bar{X}^N(T) - \bar{X}_T^{M,N} \big\|_{L^2(\Omega,\dot{H})}$ .
To this end, we introduce an auxiliary process $Y_{t_m}^{M,N}$ by
\begin{equation}
Y_{t_m}^{M,N} = E_{\tau,N}^m  X_0
    - \tau \sum_{j=1}^{m} E_{\tau,N}^{m-j+1} A  F(X^N(t_j))
	+ \sum_{j=1}^{m} E_{\tau,N}^{m-j+1}  \Delta W_j
\end{equation}
and define $\bar{Y}_{t_m}^{M,N} = Y_{t_m}^{M,N} - \mathcal{O}_{t_m}^{M,N}$.
Note that the application of an appropriate auxiliary process  was used in \cite{cui2021strongCHC,qi2020error} to deduce the strong convergence rates for the numerical approximations of
similar SPDEs.
Owning to  \eqref{eq:ass-AQ-condition}, \eqref{eq:X0}, \eqref{eq:spatial-regularity-mild-F(XN)}, \eqref{eq:E-tau-N-estimate}, \eqref{eq:E-tau-N-2-estimate} and discrete Burkholder-Davis-Gundy-type inequality, one can easily derive that for any $m \in \{0,1,,2,\cdots,M\}$,
\begin{equation}\label{eq:moment-bound-auxiliary}
\| Y_{t_m}^{M,N} \|_{L^p(\Omega,\dot{H}^3)} < \infty.
\end{equation}
Subsequently, by the triangle inequality, we have
\begin{equation}\label{XNT-XTMN}
\big\| \bar{X}^N(T) - \bar{X}_T^{M,N} \big\|_{L^2(\Omega,\dot{H})}
	\leq \big\| \bar{X}^N(T) - \bar{Y}_T^{M,N} \big\|_{L^2(\Omega,\dot{H})}
	+ \big\| \bar{Y}_T^{M,N} - \bar{X}_T^{M,N} \big\|_{L^2(\Omega,\dot{H})}.
\end{equation}
	The error term
	$\big\| \bar{X}^N(T) - \bar{Y}_T^{M,N} \big\|_{L^p(\Omega,\dot{H})}$
	can be further divided into three terms
	\begin{equation}\label{XNT-YTNM}
		\begin{split}
			\big\| & \bar{X}^N(T)  - \bar{Y}_T^{M,N}
                  \big\|_{L^p(\Omega,\dot{H})}
			=
			\Big\| ( E(T) P_N - E_{\tau,N}^M )  X_0
		\\& \quad	- \Big( \int_0^T E(T-s)P_N A  F(X^N(s)) \dd s
			  - \tau \sum_{j=1}^M E_{\tau,N}^{M-j+1} A F(X^N(t_j)) \Big)
			\Big\|_{L^p(\Omega,\dot{H})}
			\\& \leq
			\big\| (E(T) P_N  - E_{\tau,N}^M ) X_0
                            \big\|_{L^p(\Omega,\dot{H})}
			\\& \quad + \Big\| \int_0^T
			\big( E(T-s) P_N - E_{\tau,N}^{M - [s]} \big)
			A  F(X^N(s)) \dd s \Big\|_{L^p(\Omega,\dot{H})}
			\\& \quad + \Big\|  \int_0^T  E_{\tau,N}^{M - [s]} A
			\big( F(X^N(s)) - F(X^N(\lceil s \rceil)) \big) \dd s
			\Big\|_{L^p(\Omega,\dot{H})}
			\\&=:K_{11} + K_{12} + K_{13}.
		\end{split}
	\end{equation}
	By   (ii) of  Lemma \ref{lem;estimate-fully-approximation-operator} with $\beta=4$
	and Assumption \ref{assum:intial-value-data}, we deduce
	\begin{equation}
		K_{11} \leq C \, \tau |X_0|_4 \leq C \, \tau.
	\end{equation}
	Concerning the term $K_{12}$,
	by use of  (vi) and (iv) in  Lemma
	\ref{lem;estimate-fully-approximation-operator}, \eqref{eq:norm-algebra},
\eqref{eq:spatial-regularity-mild-Galerkin},
\eqref{eq:temporal-regularity-mild-Galerkin},
\eqref{eq:spatial-regularity-mild-F(XN)},
 we obtain
	\begin{equation}
		\begin{split}
			K_{12}  & \leq \Big\| \int_0^T
			\big( E(T-s) P_N- E_{\tau,N}^{M - [s]} \big)
			A P F(X^N(T)) \dd s   \Big\|_{L^p(\Omega,\dot{H})}
            \\ & \quad +  \int_0^T   \Big\|
			\big( E(T-s) P_N- E_{\tau,N}^{M - [s]} \big) A P
        ( F(X^N(s)) - F(X^N(T)) )   \Big\|_{L^p(\Omega,\dot{H})}  \dd s
			\\ & \leq C \, \tau \,  \| F(X^N(T))
                     \|_{L^{p}(\Omega,\dot{H}^2)}
			 + C \, \tau \, \int_0^T ( T - s )^{-1}
             \|P( F(X^N(s))- F(X^N(T)))\|_{L^p(\Omega,\dot{H}^2)} \dd s
            \\ &  \leq C \, \tau +  C \, \tau \,
            \Big( 1 + \sup_{r\in [0,t]} \|X^N(r)\|_{L^{4p}(\Omega,\dot{H}^2)}^2 \Big) \cdot
            \int_0^T ( T - s )^{-1}
            \|X^N(s)- X^N(T)\|_{L^{2p}(\Omega,\dot{H}^2)} \dd s
            \\ & \leq C \, \tau +  C \, \tau \,
            \int_0^T ( T - s )^{-1} ( T - s )^{\frac 14} \dd s
            \\ & \leq   C \, \tau.
		\end{split}
	\end{equation}
	To handle $K_{13}$, we decompose it into four terms with the aid of the Taylor expansion and the mild form of $X^N(t)$,
	\begin{equation}
		\begin{split}
			K_{13} & \leq
			\Big\|  \int_0^T  E_{\tau,N}^{M - [s]} A \Big( F'(X^N(s)) ( E(\lceil s \rceil -s) - I ) X^N(s) \Big) \dd s
			\Big\|_{L^p(\Omega,\dot{H})}
			\\ & \quad +
			\Big\|  \int_0^T  E_{\tau,N}^{M - [s]} A  \Big( F'(X^N(s)) \int_s^{\lceil s \rceil} \!\!\!
			E(\lceil s \rceil -r) P_N A  F(X^N(r)) \dd r \Big) \dd s
			\Big\|_{L^p(\Omega,\dot{H})}
			\\ & \quad +
			\Big\|  \int_0^T  E_{\tau,N}^{M - [s]} A  \Big( F'(X^N(s)) \int_s^{\lceil s \rceil}
			E(\lceil s \rceil -r) P_N  \dd W(r) \Big) \dd s
			\Big\|_{L^p(\Omega,\dot{H})}
			\\ & \quad +
			\Big\|  \int_0^T  E_{\tau,N}^{M - [s]} A  \Big( \int_0^1
		F'' \big(X^N(s)+\lambda( X^N (\lceil s \rceil) - X^N (s)) \big)
		\\& \qquad \qquad \qquad \big( X^N (\lceil s \rceil) - X^N (s),
			X^N (\lceil s \rceil) - X^N (s) \big)
			( 1 - \lambda)  \dd \lambda
			\Big) \dd s
			\Big\|_{L^p(\Omega,\dot{H})}
			\\&=: K_{131} + K_{132} + K_{133} +K_{134}.
		\end{split}
	\end{equation}
	The  smoothness  of $E_{\tau,N}^m$ in \eqref{eq:E-tau-N-estimate},
	\eqref{II-spatial-temporal-S(t)},
	\eqref{eq:F'--eta} and the regularity of $X^N(t)$ lead to
	\begin{equation}
		\begin{split}
			K_{131} & =
			\Big\|  \int_0^T  E_{\tau,N}^{M - [s]}
			A^{\frac32}  A^{-\frac12}
	\Big( F'(X^N(s)) ( E(\lceil s \rceil -s) - I ) X^N(s) \Big) \dd s
			\Big\|_{L^p(\Omega,\dot{H})}
			\\& \leq  C \! \int_0^T \!\!\!
			 ( T - \lfloor s \rfloor )^{ -\frac34}
			\left\|
			\big( 1 + | X^N(s) |_2^2 \big)
			\big| ( E(\lceil s \rceil -s) - I ) X^N(s) \big|_{-1}
			\right\|_{L^p(\Omega,\R)} \!\!\! \dd s
			\\& \leq C \, \tau.
		\end{split}
	\end{equation}
	Following similar approach as above and
	utilizing \eqref{eq:spatial-regularity-mild-F(XN)} yield
	\begin{equation}
		\begin{split}
			K_{132} &
			\leq  C \int_0^T ( T - \lfloor s \rfloor )^{ -\frac34}
			\Big\|
			\big( 1 + | X^N(s) |_2^2 \big)
			\int_s^{\lceil s \rceil}
		\big|  E(\lceil s \rceil -r) P_N A  F(X^N(r)) \big|_{-1} \dd r
			\Big\|_{L^p(\Omega,\R)} \dd s
			\\& \leq  C \, \tau
			\int_0^T ( T - \lfloor s \rfloor )^{ - \frac34} \dd s
			\sup_{r \in [0,T]} \| F(X^N(r)) \|_{L^{2p}(\Omega,\dot{H}^2)}
			\\& \leq C \, \tau.
		\end{split}
	\end{equation}
	From stochastic Fubini theorem, the Burkholder-Davis-Gundy-type inequality and  H\"older's inequality, it follows that
	\begin{equation}
		\begin{split}
			K_{133} & =
			\Big\|  \sum_{j=1}^M  \int_{t_{j-1}}^{t_j} \!
			\int_{t_{j-1}}^{t_j} \!\!
			\chi_{ [s,t_j) } (r)
			E_{\tau,N}^{M - [s]}
			A   F'(X^N(s)) E(\lceil s \rceil -r) P_N  \dd W(r) \dd s
			\Big\|_{L^p(\Omega,\dot{H})}
			\\&
			 =
			\Big\|  \sum_{j=1}^M  \int_{t_{j-1}}^{t_j} \!
			\int_{t_{j-1}}^{t_j}  \!\!
			\chi_{ [s,t_j) } (r)
			E_{\tau,N}^{M - [s]}
			A   F'(X^N(s)) E(\lceil s \rceil -r) P_N \dd s \dd W(r)
			\Big\|_{L^p(\Omega,\dot{H})}
			\\& \leq  \left(  \sum_{j=1}^M  \int_{t_{j-1}}^{t_j}
			\Big\|   \int_{t_{j-1}}^{t_j} \chi_{ [s,t_j) } (r)
			E_{\tau,N}^{M - [s]}
			A   F'(X^N(s)) E(\lceil s \rceil -r)  \dd s \Big\|_{L^p(\Omega,\mathcal{L}_2^0)}^2
			\, \dd r    \right)^{\frac12}
			\\& \leq C \, \tau^{\frac12}
			\left(  \sum_{j=1}^M  \int_{t_{j-1}}^{t_j} \!
			\int_{t_{j-1}}^{t_j} \!\!
			\Big\|   E_{\tau,N}^{M - [s]}
			A^{\frac12}  A^{\frac12}
			 F'(X^N(s)) E(\lceil s \rceil -r)  \Big\|_{L^p(\Omega,\mathcal{L}_2^0)}^2
			\dd s \, \dd r    \right)^{\frac12}
			\\& \leq C \, \tau^{\frac12}
			\left(  \sum_{j=1}^M  \int_{t_{j-1}}^{t_j} \!
			 \int_{t_{j-1}}^{t_j} \!\!
			( T - \lfloor s \rfloor )^{-\frac12}
			\big( 1 + \|X^N(s)\|_{L^{2p}(\Omega,\dot{H}^2)}^4 \big)
            \big\| A^{\frac12} Q^{\frac12} \big\|_{\mathcal{L}_2}^2
			\dd s \, \dd r    \right)^{\frac12}
			\\& \leq C \, \tau,
		\end{split}
	\end{equation}
where $\chi_{[0,t]}$ denotes the indicate function on $[0,t]$.
Additionally, \eqref{eq:F'-1}, \eqref{eq:E-tau-N-estimate} and the stability of $E(\lceil s \rceil -r)$ were used in the third inequality and in the last inequality we used \eqref{eq:ass-AQ-condition} and \eqref{eq:spatial-regularity-mild-Galerkin}.
Owing to H\"{o}lder's inequality, the Sobolev embedding inequality
$\dot{H}^{\delta} \subset V$ for $ \frac32 < \delta < 2$ and
Proposition \ref{prop:regularity-spatial-spectral},
we obtain
\begin{equation}
\begin{split}
K_{134} & = \Big\|  \int_0^T  E_{\tau,N}^{M - [s]} A
  \Big( \int_0^1
	F'' \big(X^N(s)+\lambda( X^N (\lceil s \rceil) - X^N (s)) \big)
	\\& \qquad \qquad \qquad \big( X^N (\lceil s \rceil) - X^N (s),
	X^N (\lceil s \rceil) - X^N (s) \big)
		( 1 - \lambda)  \dd \lambda \Big) \dd s
			\Big\|_{L^p(\Omega,\dot{H})}
	\\& \leq C  \int_0^T
		( T - \lfloor s \rfloor )^{ -\frac{2+\delta}4}
		\| X^N (\lceil s \rceil)  -  X^N (s)
             \|_{L^{4p}(\Omega,\dot{H})}^2
		\big( 1 +  \sup_{s \in [0,T]}
		 \| X^N (s) \|_{L^{2p}(\Omega,V)} \big) \dd s
	\\& \leq C \, \tau.
\end{split}
\end{equation}
Gathering the above estimates and \eqref{XNT-YTNM} together gives
\begin{equation}\label{eq:K2-first-term}
\big\| \bar{X}^N(T) - \bar{Y}_T^{M,N}
   \big\|_{L^p(\Omega,\dot{H})}
		\leq C \, \tau.
\end{equation}
Finally, we turn to remaining  error term
$\big\| \bar{Y}_T^{M,N} - \bar{X}_T^{M,N} \big\|_{L^2(\Omega,\dot{H})}$  in  \eqref{XNT-XTMN}.
Denoting
	$e_t^{M,N}=\bar{Y}_t^{M,N} - \bar{X}_t^{M,N}$,
	we have
	\begin{equation}
		e_{t_m}^{M,N} - e_{t_{m-1}}^{M,N}
		+ \tau A^2 e_{t_m}^{M,N}
		=  \tau P_N A  F(X_{t_m}^{M,N})  -  \tau P_N A F(X^N(t_m)).
	\end{equation}
	Multiplying both sides by $A^{-1} e_{t_m}^{M,N}$ shows
	\begin{equation}
		\begin{split}
			\langle e_{t_m}^{M,N} - & e_{t_{m-1}}^{M,N},
			A^{-1}  e_{t_m}^{M,N} \rangle
			+ \tau \langle A^2 e_{t_m}^{M,N}, A^{-1} e_{t_m}^{M,N} \rangle
			\\& = \tau \langle
			F( \bar{X}_{t_m}^{M,N} + \mathcal{O}_{t_m}^{M,N} )
			- F( \bar{Y}_{t_m}^{M,N} + \mathcal{O}_{t_m}^{M,N} ),
			e_{t_m}^{M,N} \rangle
			\\& \quad +   \tau \langle
			F( \bar{Y}_{t_m}^{M,N} + \mathcal{O}_{t_m}^{M,N} )
			-  F( \bar{X}^N(t_m) + \mathcal{O}_{t_m}^{M,N} ),
			e_{t_m}^{M,N} \rangle
			\\& \quad +   \tau \langle
			F( \bar{X}^N(t_m) + \mathcal{O}_{t_m}^{M,N} )
			-  F( \bar{X}^N(t_m) + \mathcal{O}_{t_m}^{N} ),
			e_{t_m}^{M,N} \rangle.
		\end{split}
	\end{equation}
Following similar approach
in \eqref{eq:space-frechet-estimate} and using the inequality $\langle e_{t_m}^{M,N} -  e_{t_{m-1}}^{M,N},
A^{-1}  e_{t_m}^{M,N} \rangle
\geq  \tfrac12 \big( |e_{t_m}^{M,N}|_{-1}^2 -
|e_{t_{m-1}}^{M,N}|_{-1}^2 \big)$
and the monotonicity of $F$ in \eqref{eq:one-side-condition},
we further obtain
\begin{equation}
\begin{split}
\tfrac12 \big( |e_{t_m}^{M,N}|_{-1}^2   -
	|e_{t_{m-1}}^{M,N}|_{-1}^2 \big)
	+ \tau | e_{t_m}^{M,N} |_1^2
	& \leq  \tfrac34 \tau | e_{t_m}^{M,N} |_1^2
	+ \tfrac98 \tau | e_{t_m}^{M,N} |_{-1}^2
	\\& \quad + C \, \tau \big\|
      F( \bar{Y}_{t_m}^{M,N} + \mathcal{O}_{t_m}^{M,N} )
		-  F( \bar{X}^N(t_m) + \mathcal{O}_{t_m}^{M,N} ) \big\|^2
	\\& \quad + C \, \tau
      \big| F( \bar{X}^N(t_m) + \mathcal{O}_{t_m}^{M,N} )
		-  F( \bar{X}^N(t_m) + \mathcal{O}_{t_m}^{N} ) \big|_{-1}^2.
\end{split}
\end{equation}
By iteration  in $m$ and  Gronwall's inequality, we obtain
	\begin{equation}
		\begin{split}
			| e_{T}^{M,N} |_{-1}^2
			+  \tau \sum_{j=1}^M | e_{t_j}^{M,N} |_1^2
			& \leq C \, \tau \sum_{j=1}^M \Big(
			\big\| F( \bar{Y}_{t_j}^{M,N} + \mathcal{O}_{t_j}^{M,N} )
		-  F( \bar{X}^N(t_j) + \mathcal{O}_{t_j}^{M,N} ) \big\|^2 \Big)
			\\& \quad
			+ C \, \tau \sum_{j=1}^M \Big(
         \big| F( \bar{X}^N(t_j) + \mathcal{O}_{t_j}^{M,N} )
			-  F( \bar{X}^N(t_j) + \mathcal{O}_{t_j}^{N} ) \big|_{-1}^2
			\Big).
		\end{split}
	\end{equation}
It is worth  mentioning that \eqref{eq:K2-first-term} also holds for arbitrary $t_j,j \in \{1,\cdots,M\}$ by repeating the same argument from \eqref{XNT-YTNM} to \eqref{eq:K2-first-term}.
Then, employing  \eqref{eq:local-condition}, \eqref{eq:spatial-regularity-mild-Galerkin}, \eqref{eq:F'--eta}, \eqref{eq:moment-bound-auxiliary} and Lemma \ref{lem:ON-OMN} results in
	\begin{equation}\label{eq:int-full-estimate}
		\begin{split}
			\Big\|  \tau  \sum_{j=1}^M | e_{t_j}^{M,N} |_1^2
			\Big\|_{L^p(\Omega,\R)}
			& \leq C ~\tau \sum_{j=1}^M \! \Big\|
			\big\| \bar{Y}_{t_j}^{M,N}\! -\! \bar{X}^N(t_j) \big\|^2
			\big( 1 \! + \! \|\bar{Y}_{t_j}^{M,N}\|_V^4
			+  \|\bar{X}^N(t_j)\|_V^4
			+  \|\mathcal{O}_{t_j}^{M,N}\|_V^4 \big)
			\Big\|_{L^p(\Omega,\R)}
			\\& \quad
			+ C ~\tau \sum_{j=1}^M \!\Big\|
			\big| \mathcal{O}_{t_j}^{M,N}
			-  \mathcal{O}_{t_j}^{N} \big|_{-1}^2
			\big( 1 \!+\! |\bar{X}^N(t_j)|_2^4
			+  |\mathcal{O}_{t_j}^{M,N}|_2^4
			+  |\mathcal{O}_{t_j}^{N}|_2^4 \big)
			\Big\|_{L^p(\Omega,\R)}
			\\& \leq C \, \tau^{2}.
		\end{split}
	\end{equation}
	Furthermore,  since
\begin{equation}
e_T^{M,N} = \bar{Y}_T^{M,N} - \bar{X}_T^{M,N}
 = \tau \sum_{j=1}^M E_{\tau,N}^{M-j+1} A
			\big(F(X_{t_j}^{M,N}) - F(X^N(t_j)) \big),
\end{equation}
we split $\| e_T^{M,N} \|_{L^p(\Omega,\dot{H})}$ into three parts
	\begin{equation}
		\begin{split}
			\|  e_T^{M,N}  \|_{L^p(\Omega,\dot{H})}
		& = \tau \Big\|
			\sum_{j=1}^M E_{\tau,N}^{M-j+1} A
			\big( F(X^N(t_j)) - F(X_{t_j}^{M,N}) \big)
			\Big\|_{L^p(\Omega,\dot{H})}
			\\& \leq  \tau \sum_{j=1}^M \Big\|
			E_{\tau,N}^{M-j+1}  A
			\big( F( \bar{X}^N(t_j) + \mathcal{O}_{t_j}^N)
			- F(\bar{Y}_{t_j}^{M,N} + \mathcal{O}_{t_j}^N ) \big)
			\Big\|_{L^p(\Omega,\dot{H})}
			\\& \quad + \tau \sum_{j=1}^M \Big\|
			E_{\tau,N}^{M-j+1} A
			\big(  F(\bar{Y}_{t_j}^{M,N} + \mathcal{O}_{t_j}^N )
			- F( \bar{Y}_{t_j}^{M,N} + \mathcal{O}_{t_j}^{M,N} ) \big)
			\Big\|_{L^p(\Omega,\dot{H})}
			\\& \quad + \tau \Big\| \sum_{j=1}^M
			E_{\tau,N}^{M-j+1}  A
			\big(  F( \bar{Y}_{t_j}^{M,N} + \mathcal{O}_{t_j}^{M,N} )
			- F( \bar{X}_{t_j}^{M,N} + \mathcal{O}_{t_j}^{M,N} ) \big)
			\Big\|_{L^p(\Omega,\dot{H})}
			\\&=:   Err_1 + Err_2 + Err_3.
		\end{split}
	\end{equation}
	Taking \eqref{eq:E-tau-N-estimate}, \eqref{eq:local-condition},
	\eqref{eq:K2-first-term}, H\"{o}lder's inequality and moment bounds of $Y_{t_m}^{M,N}$ and $X^N(t)$ into account,
	we arrive at
	\begin{equation}
		\begin{split}
			Err_1
			& \leq C \, \tau \sum_{j=1}^M t_{M-j+1}^{-\frac12}
			\|\bar{X}^N(t_j) - \bar{Y}_{t_j}^{M,N}\|_{L^{2p}(\Omega,\dot{H})}
			\\& \qquad  \qquad
			\big( 1+ \|\bar{X}^N(t_j)\|_{L^{4p}(\Omega,V)}^2
			+ \|\bar{Y}_{t_j}^{M,N}\|_{L^{4p}(\Omega,V)}^2
			+ \|\mathcal{O}_{t_j}^N\|_{L^{4p}(\Omega,V)}^2 \big)
			\\& \leq C \, \tau \, \sum_{j=1}^M t_{M-j+1}^{-\frac12} \,\tau
		 \leq C \, \tau.
		\end{split}
	\end{equation}
	Analogously  to  the   above estimate but with \eqref{eq:F'--eta} instead,
	we derive
	\begin{equation}
		\begin{split}
			Err_2
			& \leq C \tau \sum_{j=1}^M   t_{M-j+1}^{-\frac34}
			\|\mathcal{O}_{t_j}^N - \mathcal{O}_{t_j}^{M,N}
			\|_{L^{2p}(\Omega,\dot{H}^{-1})}
			\\& \quad
			\big( 1+
              \|\bar{Y}_{t_j}^{M,N}\|_{L^{4p}(\Omega,\dot{H}^2)}^2
			+ \|\mathcal{O}_{t_j}^N \|_{L^{4p}(\Omega,\dot{H}^2)}^2
			+ \|\mathcal{O}_{t_j}^{M,N} \|_{L^{4p}(\Omega,\dot{H}^2)}^2
                \big)
			\\& \leq
			C \, \tau.
		\end{split}
	\end{equation}
	At last, combining \eqref{eq:F'-1}, H\"{o}lder's inequality, \eqref{eq:int-full-estimate} and regularity of
	$Y_{t_m}^{M,N}$ and $X_{t_m}^{M,N}$ leads to
	\begin{equation}
		\begin{split}
			Err_3
			& \leq C
			\left\| \tau \sum_{j=1}^M   t_{M-j+1}^{-\frac14}
			\big| F( \bar{Y}_{t_j}^{M,N} + \mathcal{O}_{t_j}^{M,N} )
			- F( \bar{X}_{t_j}^{M,N} + \mathcal{O}_{t_j}^{M,N} ) \big|_1
			\right\|_{L^p(\Omega,\R)}
			\\& \leq C
			\left\|
			\tau \sum_{j=1}^M   t_{M-j+1}^{-\frac14}
			|e_{t_j}^{M,N}|_1
			\big( 1 + |\bar{Y}_{t_j}^{M,N}|_2^2
			+ |\bar{X}_{t_j}^{M,N}|_2^2
			+ |\mathcal{O}_{t_j}^{M,N}|_2^2   \big)
			\right\|_{L^p(\Omega,\R)}
			\\&\leq C
			\left\|
			\tau \sum_{j=1}^M
			|e_{t_j}^{M,N}|_1^2
			\right\|_{L^p(\Omega,\R)}^{\frac12}
			  \times
			\left\|
			\tau \sum_{j=1}^M   t_{M-j+1}^{-\frac12}
			\big( 1 + |\bar{Y}_{t_j}^{M,N}|_2^4
			+ |\bar{X}_{t_j}^{M,N}|_2^4
			+ |\mathcal{O}_{t_j}^{M,N}|_2^4  \big)
			\right\|_{L^p(\Omega,\R)}^{\frac12}
			\\& \leq C \, \tau.
		\end{split}
	\end{equation}
Combining  the above estimates together yields
\begin{equation}\label{time1}
\big\| \bar{X}^N(T) - \bar{X}_T^{M,N} \big\|_{L^2(\Omega,\dot{H})} \leq
C \, \tau,
\end{equation}
and thus $|K_1|\le C \, \tau$.
The estimate of $K_2$ relies on a second-order Taylor expansion
 and the triangle inequality:
	\begin{equation}
		\begin{split}
			|K_2| &
           = \Big| \E \big[ \Phi (\bar{X}^N(T) + \mathcal{O}_T^N) \big]
			- \E \big[ \Phi (\bar{X}^N(T) + \mathcal{O}_T^{M,N}) \big]
			\Big|
			\\& \leq
			\Big| \E \Big[ \Phi'(X^N(T))
			\big( \mathcal{O}_T^{M,N} -  \mathcal{O}_T^{N} \big)
                 \Big] \Big|
			\\&\quad
            + \Big| \E \Big[ \int_0^1 \Phi'' \big(  X^N(T) + \lambda
			( \mathcal{O}_T^{M,N} -  \mathcal{O}_T^{N} )   \big)
			\big( \mathcal{O}_T^{M,N} -  \mathcal{O}_T^{N},
			\mathcal{O}_T^{M,N} -  \mathcal{O}_T^{N}  \big)
			( 1 - \lambda )  \dd \lambda \Big]  \Big|
	\\& \leq \Big| \E \big[ \Phi'(X^N(T) )
	\big( \mathcal{O}_T^{M,N} -  \mathcal{O}_T^{N} \big) \big] \Big|
	+ C \, \E \big[ \|\mathcal{O}_T^{M,N} - \mathcal{O}_T^{N}\|^2 \big].
		\end{split}
	\end{equation}
Thanks to  Lemma \ref{lem:ON-OMN} with $\beta=0$, we have
\begin{equation}\label{time2}
\E \big[ \|\mathcal{O}_T^{M,N} -  \mathcal{O}_T^{N}\|^2 \big]
  \leq ( C \tau^{\frac 34} )^2 \leq C \, \tau^{\frac 32}.
\end{equation}
Then, we turn our attention to the first term,
\begin{equation}
\begin{split}
\Big| \E  \big[ \Phi'(X^N(T) )
\big( \mathcal{O}_T^{M,N} -  \mathcal{O}_T^{N} \big) \big] \Big|
&  = \Big| \E  \Big[ \Big \langle \int_0^T
    \big( E(T-s) P_N - E_{\tau,N}^{M-[s]} \big) \dd W(s),
   \Phi^{'} (X^N(T)) \Big \rangle \Big]  \Big|
 \\& = \Big| \E \int_0^T \big \langle E(T-s)P_N-E_{\tau,N}^{M-[s]},
	\mathcal{D}_s  \Phi'(X^N(T)) \big\rangle_{\mathcal{L}_2^0} \dd s \Big|
 \\ & \leq \E \int_0^T \big\| E(T-s)P_N-E_{\tau,N}^{M-[s]} \big\|_{\mathcal{L}_2^0} \| \Phi''(X^N(T)) \mathcal{D}_s X^N(T) \|_{\mathcal{L}_2^0} \dd s
 \\ & \leq C \int_0^T \big\| \big( E(T-s)P_N-E_{\tau,N}^{M-[s]} \big) A^{-\frac12} \big\|_{\mathcal{L}} \|A^{\frac12}Q^{\frac12}\|_{\mathcal{L}_2} \dd s
 \\ & \leq C \tau \int_0^T ( T-s )^{-\frac 34} \dd s
	\leq C \, \tau,
\end{split}
\end{equation}
where \eqref{eq:ass-AQ-condition}, the Malliavin integration by parts formula \eqref{Malliavin integration by parts},  \eqref{eq:malliavin-derivative} in Proposition	\ref{Estimate of Malliavin derivative of the solution}
and (iv) in Lemma \ref{lem;estimate-fully-approximation-operator} with $\mu = 1$ were used.
Therefore, we obtain $|K_1| \leq C\, \tau$ and $|K_2| \leq C\, \tau$.
The proof is thus complete.
\end{proof}

\begin{corollary}\label{coro:strong}
	As a by-product of the weak error analysis, one can easily obtain the rates of the strong error, for $N \in \N $ and $m \in \{ 1,2,\cdots,M \}$,
\begin{equation}
\begin{split}
\|X(t_m) - X_{t_m}^{M,N} \|_{L^2(\Omega,\dot{H})}
& \leq \|\bar{X}(t_m) - \bar{X}_{t_m}^{M,N}\|_{L^2(\Omega,\dot{H})}
+ \|\mathcal{O}_{t_m}-\mathcal{O}_{t_m}^{M,N}\|_{L^2(\Omega,\dot{H})}
\\ &
\leq \|\bar{X}(t_m) - \bar{X}^N(t_m)\|_{L^2(\Omega,\dot{H})}
 + \|\bar{X}^N(t_m) - \bar{X}_{t_m}^{M,N}\|_{L^2(\Omega,\dot{H})}
\\ &  \quad
+ \|\mathcal{O}_{t_m}-\mathcal{O}_{t_m}^{N}\|_{L^2(\Omega,\dot{H})}
+\|\mathcal{O}_{t_m}^{N}-\mathcal{O}_{t_m}^{M,N}\|_{L^2(\Omega,\dot{H})}
\\ &  \leq C ( \lambda_N^{-2} + \tau +
     \lambda_N^{-\frac32} + \tau^{\frac34})
\\ & \leq C ( \lambda_N^{-\frac32} + \tau^{\frac34}),
\end{split}
\end{equation}
where the third inequality follows from \eqref{space1}, \eqref{time1}, \eqref{space2} and \eqref{time2} with $t_m$ instead of $T$, successively.
The strong error estimates here, the same as that in \cite{cui2021strongCHC,qi2021existence,qi2020error}, coincide with the spatial regularity of $X(t)$, and thus are optimal.
\end{corollary}

\begin{remark}
It is worthwhile to mention that the obtained weak convergence rate in time (i.e., $\mathcal{O}(\tau)$) is optimal for the Euler--type method applying to stochastic differential equation.
\end{remark}

\section*{Acknowledgements}

M. Cai and S. Gan are supported by NSF of China (No. 11971488).
M. Cai is supported by the China Scholarship Council.
Y. Hu is supported  by an NSERC discovery grant.
We are very grateful to the referees for the interesting and constructive comments and suggestions.


\end{document}